\documentclass[11pt, a4paper]{amsart}
\usepackage{amsmath}  % for align. HAS TO BE LOADED IN THE BEGINNING
\usepackage{amsthm}   % for newtheorem
\usepackage{amssymb}  % arrows and stuff
\usepackage{amsopn}	 %\DeclareMathOperator
\usepackage{amsrefs}	 %Bibtex replacement
\usepackage{subfiles} % to compile child documents
\usepackage{standalone} %to ignore preamble of subfile
\usepackage{xspace}   % for xspace command
\usepackage{tikz-cd}
\usepackage{xifthen}  % for \ifthenelse
\usepackage{stmaryrd} % for \varowedge
\usepackage[sc,osf]{mathpazo} % a nice font
\usepackage{upgreek}  % upright greek symbols
\usepackage{setspace}	 %Change spaceing in document
\usepackage{enumerate} %for changing numerations
\usepackage{tensor}
\usetikzlibrary{matrix,arrows,decorations.pathmorphing}
\usepackage{geometry}
\usepackage{hyperref} %for links in chapters and sections
\swapnumbers%1.4 Theorem instead of Theorem 1.4
\theoremstyle{plain}
\newtheorem{thm}{Theorem}[section]
\newtheorem{lem}[thm]{Lemma}
\newtheorem{cor}[thm]{Corollary}
\newtheorem{prop}[thm]{Proposition}
\theoremstyle{definition}
\newtheorem{defi}[thm]{Definition}
\theoremstyle{remark}
\newtheorem{rem}[thm]{Remark}

\DeclareMathOperator{\R}{\mathbb{R}}
\DeclareMathOperator{\C}{\mathbb{C}}

\DeclareMathOperator{\Ric}{Ric}

\DeclareMathOperator{\Rm}{\mbox{Rm}}
\DeclareMathOperator{\cop}{\mathfrak{Rm}}
\DeclareMathOperator{\pl}{\mbox{pl}}

\begin{document}

\title[The branching curves and their appl. to the  $4$-dimensional Ricci flow]{The branching curves and their application to the four dimensional Ricci flow}
\author{Ilias Tergiakidis}
\thanks{This work has been funded by the RTG 1493 in the Mathematisches Institut, Georg-August Universit\"at G\"ottingen and is part of my Ph.D. dissertation \cite{Tergiakidis}. I would like to thank my supervisor Victor Pidstrygach.}
\address{Department of Mathematics, Aristotle University of Thessaloniki, 54124 Thessaloniki, Greece}
\email{tergiakidis@math.auth.gr}
\begin{abstract}
We study the four dimensional Ricci flow with the help of local invariants. If $(M^4,g(t))$ is a solution to the Ricci flow and $x \in M^4$, we can associate to the point $x$ a one-parameter family of curves, which lie in the product of two projective lines. This allows us to reformulate the Cheeger-Gromov-Hamilton Compactness Theorem in the context of these curves. We use this result, in order to study Type I singularities in dimension four and give a characterization of the corresponding singularity models.
\end{abstract}
\maketitle

\section{Introduction}

Let $(M^n,g_0)$ be a smooth Riemannian manifold. Hamilton's Ricci flow, introduced in \cite{Hamilton}, is a PDE that describes the evolution of the Riemannian metric tensor:
\begin{eqnarray*}
\frac{\partial}{\partial t}g(t) &=& -2 \mbox{Ric}_{g(t)}\\
g(0) &= & g_0,
\end{eqnarray*}
where $g(t)$ is a one-parameter family of metrics on $M^n$ and $\mbox{Ric}_{g(t)}$ denotes the Ricci curvature with respect to $g(t)$. The minus sign makes the Ricci flow a heat-type equation, so it is expected to "average out" the curvature.

A triple  $(M^n,g,f)$  is called a gradient shrinking Ricci soliton, if there exists a gradient vector field $X= \nabla^g f =\mbox{grad} f$ for some $f \in C^{\infty}(M)$ (called the potential function) and $\kappa>0$, such that
\begin{equation*}\label{eq: gradient Ricci soliton}
\Ric_g + \nabla^g \nabla^g f =  \kappa g.
\end{equation*}
Ricci solitons give rise to special solutions to the Ricci flow.  A gradient shrinking Ricci soliton satisfying the equation 
$$\mbox{Ric}_{g_0}+ \nabla^{g_0}\nabla^{g_0} f_0=\kappa g_0$$
corresponds to the self similar solution
$$g(t)=(1-2 \kappa t) \phi(t)^{*} (g_0),$$
where $\phi(t)$ is the one-parameter family of diffeomorphisms generated by the one-parameter family of vector fields $X(t)=\frac{\nabla^{g_0} f_0}{1-2 \kappa t}$. 

For a gradient shrinking Ricci soliton it is always possible to rescale the metric by $2 \kappa$ and shift the function $f_0$ by the constant $-C_0$, so that the soliton equation becomes
$$\mbox{Ric}_{g_{0}}+ \nabla^{g_{0}}\nabla^{g_{0}} f_{0}=\frac{1}{2} g_{0.}$$
We call such a soliton a normalized gradient shrinking Ricci soliton. 

We say that the gradient shrinking Ricci soliton is complete if $(M^n,g_0)$ is complete and the vector field $\nabla^{g_0}f_0$ is complete. 

We next introduce the canonical form for the associated time-dependent version of a normalized gradient shrinking Ricci soliton.  Let $(M^n,g_0,f_0)$ be a complete normalized gradient shrinking Ricci soliton. Then there exists a solution $g(t)$ of the Ricci flow with $g(0)=g_0$, diffeomorphisms $\phi(t)$ with $\phi(0)=\mbox{id}_M$, functions $f(t)$ with $f(0)=f_0$ defined for all $t$ with $\sigma(t)=1-t>0,$ such that the following hold:
\begin{enumerate}[(i)]
\item $\phi(t):M^n \to M^n$ is the one-parameter family of diffeomorphisms generated by $X(t)=\frac{\nabla^{g_0} f_0}{1-t}$,
\item $g(t)=\sigma(t)\phi(t)^* (g_0)$ on $(-\infty,1)$,
\item $f(t)=f_0 \circ \phi(t)=\phi(t)^* (f_0)$.
\end{enumerate}
Futhermore,
\begin{eqnarray*}
\mbox{Ric}_{g(t)}+\nabla^{g(t)}\nabla^{g(t)} f(t) &=& \frac{1}{2(1-t)}g(t),\\
\frac{\partial}{\partial t}f(t)&=&|\nabla^{g(t)} f(t)|^{2}_{g(t)}. 
\end{eqnarray*}

The classification of $3$-dimensional gradient shrinking Ricci solitons was done by the works of Perelman \cite{Per2}, Naber \cite{Naber}, Ni-Wallach \cite{NW} and Cao-Chen-Zhu \cite{CCZ}. They showed that a $3$-dimensional gradient shrinking Ricci soliton is a quotient of either $\mathbb{S}^3$ or $\mathbb{R}^3$ or $\mathbb{S}^2 \times \mathbb{R}$. This means that the only noncompact nonflat $3$-dimensional gradient shrinking Ricci solitons are the round cylinder and its quotients.  In this paper we will focus on the $4$-dimensional gradient shrinking Ricci solitons. In dimension $4$ there is no full classification of the gradient shrinking Ricci solitons. There is some classification done under curvature assumptions by Ni-Wallach \cite{NW1} and Naber \cite{Naber}.

The Ricci flow is a type of nonlinear heat equation for the metric and it is expected, that it develops singularities. We will focus on finite time singularities and $T < \infty$ will denote the singular time.  Even more specifically, a complete solution $(M^n,g(t))$ to the Ricci flow defined on a finite time interval $[0,T)$, $T < \infty$ is called a Type I Ricci flow if there exists some constant $C>0$ such that
$$\sup_M |\Rm_{g(t)}|_{g(t)} \leq \frac{C}{T-t},$$
 for all $t \in [0,T)$. In such a case, we say that the solution $g(t)$ develops a Type I singularity at time $T$. The most well known examples of Type I singularities are the neckpinch singularity modelled on a shrinking cylinder and those modelled on flows starting at a positive Einstein metric or more general at a gradient shrinking Ricci soliton with bounded curvature. 

A  sequence of points and times $\{ (x_i,t_i) \}$  with $x_i \in M^n$ and $t_i \to T$ is called an essential blow up sequence if there exists a constant $c>0$ such that
$$|\Rm_{g(t_i)}|_{g(t_i)}(x_i) \geq \frac{c}{T-t_i}.$$
A point $x\in M^n$ in a Type I Ricci flow is called a Type I singular point if there exists an essential blow up sequence with $x_i \to x$ on $M^n$. The set of all Type I singular points is denoted by $\Sigma_I$. 

In order to study finite time singularities one should take parabolic rescalings of the solutions about sequences of points and times, where the time tends to the singularity time $T$. The limit solutions of such sequences, if they exist, are ancient solutions and are called singularity models. Hamilton conjectured in \cite{Hamilton2}, that a suitable blow up sequence for a Type I singularity converges to a nontrivial gradient shrinking Ricci soliton. Sesum \cite{Sesum} confirmed the conjecture in the case where the blow up limit is compact. In the general case, blow up to a gradient shrinking soliton was proved by Naber \cite{Naber}. However, it remained an open question whether the limit soliton Naber constructed is nontrivial. Enders, M\"uller and Topping eliminated this possibility in \cite{EMT}.

Understanding the formation and the nature of singularities is a very crucial step. This can be done by classifying the set of singularitiy models that may arrise. We focus on dimension four. M\'aximo showed in \cite{Maximo}, that in dimension four, the singularity models for finite singularities can have Ricci curvature of mixed sign. Thus the only restriction remaining is the nonnegativity of the scalar curvature. This is unfortunately not the best scenario, because this condition is too week, in order to obtain a full classification result for the singularity models in dimension four. The experts believe, that the best alternative would be to classify the generic or at least the stable singularity models. A singularity model developing certain original data is labeled stable, if flows starting from all sufficient small perturbations of that data develop singularities with the same singularity model. Furthermore, a singularity model is labeled generic, if flows that start from an open dense subset of all possible initial data develop singularities having the same singularity model. Clearly, a singularity model can be generic only if it is stable. More details can be found in \cite{IKS}. Furthermore, it is conjectured by experts, that the only candidates for generic singularity models in dimension four are $\mathbb{S}^4$, $\mathbb{S}^3 \times \mathbb{R}$, $\mathbb{S}^2 \times \mathbb{R}^2$. These singularity models are known to be generic. There is another soliton, which is not known yet to be generic or not. This the  $(\mathcal{L}_{-1}^{2},h)$, which is the blow down soliton constructed by Feldman, Ilmanen and Knopf in \cite{FIK}. If the blow down soliton is generic, then it should be also in the list above.

In this paper we try to contribute in the direction of understanding the $4$-dimensional gradient shrinking Ricci solitons, which can appear as singularity models for Type I singularities. This is done by considering local invariants for a $4$-dimensional Riemannian manifold and trying to interpret the limiting solitons in the language of these local invariants. Let's be more precise.

In Section 2 we describe a construction of A. N. Tyurin. Tyurin showed in \cite{Tyurin}, that for any $4$-dimensional Riemannian manifold $(M^4,g)$ and fixed point $x \in M$, one can define in a natural way three quadratic forms in  $\Lambda^2 T_xM$. These are given by the exterior power evaluated at a volume form, the second exterior power of the Riemannian metric $g$ and the curvature tensor of the Riemannian connection respectively. After complexifying, their projectivization defines three quadrics in $\mathbb{P}(\Lambda^2 T_xM \otimes \mathbb{C})$. For any point $x\in M$ at which the quadratic forms are linearly independent, the intersection of these three quadrics defines a singular $K3$ surface. After performing a resolution of the singular points, the resolved $K3$ is a double branched cover of  a smooth quadric in  $\mathbb{P}(T_xM \otimes \mathbb{C})$. In many cases the branching locus corresponds to a curve of bidegree $(4,4)$ in the product of two projective lines. The branching curve denoted by $\Gamma_x$ will be our local invariant for the $4$-dimensional manifold $M$. Its coefficients will be determined by the components of the Riemann curvature tensor. Note that four years later, V. V. Nikulin in \cite{Nikulin} extended the result to the case of pseudo-Riemannian manifolds with a Lorentz metric. 

In Section 3 we do some explicit calculations and compute examples of branching curves (local invariants) for some $4$-dimensional gradient shrinking Ricci solitons.

In Section 4 we prove the main result of this paper, namely Theorem  \ref{th: converge of curves}. It states, that in dimension four, convergence of marked solutions to the Ricci flow (convergence in the Cheeger-Gromov sence) implies convergence for branching curves. The proof can be found in Section 4.

\begin{thm} \label{th: converge of curves}
Let $\{ (M^4, g_i(t),x, F_i(t)) \}_{i \in \mathbb{N}}$, $t \in (\alpha, \omega) \ni 0$ be a sequence of smooth, complete, marked solutions to the Ricci flow, where the time-dependet frame $F_i(t)$ evolves to stay orthonormal. Assume, that the sequence  converges to a complete marked solution to the Ricci flow $$(M_{\infty}^{4}, g_{\infty}(t),x_{\infty}, F_{\infty}(t)), \, t \in (\alpha, \omega) \mbox{ as } i \to \infty,$$
where $F_{\infty}(t)$ evolves to stay orthonormal as well. Let $\{ \Gamma_{x}^{g_i(t)} \}_{i \in \mathbb{N}}$ be the sequence of one-parameter families of branching curves associated to $x \in M$ and  $\Gamma_{x_{\infty}}^{g_\infty (t)}$ the one-parameter family of branching curves associated to $x_{\infty} \in M_{\infty}$ (if this exists). Then $\Gamma_{x}^{g_i(t)}$ converges to $\Gamma_{x_{\infty}}^{g_\infty (t)}$ as $i \to \infty$, in the sense that the coefficients of the curves converge.
\end{thm}

We use the previous Theorem and combine it with the result of Enders, M\"uller and Topping \cite{EMT} in mentioned above, in order to obtain a characterization of the gradient shrinking Ricci solitons, which can appear as singularity models for Type I singularities in dimension four. We call this result Corollary \ref{cor: ilias}. The proof of this Corollary can be found in Section 4.

\begin{cor}\label{cor: ilias}
Let $(M^4,g(t))$ be a Type I Ricci flow on $[0,T)$ and $x \in \Sigma_I$. Furthermore let  $ \Gamma_{x}^{g(t)}$ be the one-paramater family of branching curves associated to  $x$. Let us choose a sequence of scaling factors $\lambda_i$, such that $\lambda_i \to 0$. We define the rescaled Ricci flows $(M^4,g_i(t),x,F_i(t))$ by
$$g_i(t)=\lambda_{i}^{-1}g(T+\lambda_i t),\, t \in [-\lambda_{i}^{-1}T,0),$$
where  the time-dependet frame $F_i(t)$ evolves to stay orthonormal. Then the one-parameter family of curves $\Gamma_{x}^{g_i(t)}$ is $\Gamma_{x}^{g(T+\lambda_i t)}$ and subconverges to the one-parameter family of curves $\Gamma_{x_{\infty}}^{g_{\infty}(t)}$ (if this exists) of a nontrivial normalized gradient shrinking Ricci soliton $(M_{\infty}^{n}, g_{\infty}(t),x_{\infty}, F_{\infty}(t))$, $t \in (-\infty,0) $ in canonical form, where $F_{\infty}(t)$ evolves to stay orthonormal. 
\end{cor}

\section{A local invariant of a four dimensional Riemannian manifold}

\subsection{The geometry of three quadrics in $\mathbb{P}(\Lambda^2 T_x M \otimes \C)$}

Let $(M,g)$ be a four dimensional Riemannian manifold. We denote by $T_x M$ the tangent space at the point $x \in M$. We are going to define three quadratics forms on $\Lambda^2 T_{x} M$.

\centerline{\textbf{The quadratic form $v_x$ }:}
We define the map  
\begin{eqnarray*}
\Lambda^2 T_xM \times \Lambda^2 T_xM  &\to & \Lambda^4 T_x M \\
(u,h) &\mapsto & u \wedge h.
\end{eqnarray*}
Recall, that the volume form $\mbox{vol}_M$ on $M$ is a nowhere vanishing section of $\Lambda^4 T_{x}^{*}M$. We identify $ \Lambda^4 T_x M$ with $\R$ by evaluating $u \wedge h$ on the volume form, i.e. $\mbox{vol}_M( u \wedge h)$. So we obtain a bilinear form $v_x: \Lambda^2 T_xM \times \Lambda^2 T_xM \to \R$. 

Let now $\{ x^i \}_{i=1}^{4}$ denote local coordinates around $x$, such that $\{\frac{\partial}{\partial x^i}\}_{i=1}^{4}$ is a basis for $T_x M$ and $\{dx^ i\}_{i=1}^{4}$ is the dual to it. Then $\{\frac{\partial}{\partial x^i} \wedge \frac{\partial}{\partial x^j} \}_{1\leq i <j \leq 4}$ and $\{dx^i \wedge dx^j\}_{1\leq i <j \leq 4}$, are bases for $\Lambda^2 T_x M$ and $(\Lambda^2 T_x M)^* \simeq \Lambda^2 T_{x}^{*} M$ respectively. Let $u,h \in \Lambda^2 T_x M$ be given by 
\begin{equation}\label{eq: u}
u=\underset{1 \leq i < j \leq 4}\sum u^{ij} \frac{\partial}{\partial x^i} \wedge \frac{\partial}{\partial x^j}
\end{equation}
and 
\begin{equation}\label{eq: h}
h=\underset{1 \leq i < j \leq 4}\sum h^{ij} \frac{\partial}{\partial x^i} \wedge \frac{\partial}{\partial x^j}
\end{equation}
with respect to this basis. Then
\begin{eqnarray*}\label{eq: map}
\Lambda^2 T_xM \times \Lambda^2 T_xM  &\to & \Lambda^4 T_x M  \nonumber \\
(u,h) &\mapsto & (u^{12}h^{34}-u^{13}h^{24}+u^{14}h^{23}\\ \nonumber
& & +u^{23} h^{14}- u^{24}h^{13}+u^{34}h^{12}) \frac{\partial}{\partial x^1} \wedge \frac{\partial}{\partial x^2} \wedge \frac{\partial}{\partial x^3} \wedge \frac{\partial}{\partial x^4}.
\end{eqnarray*}
Recall, that the Riemannian volume form is given by  $\sqrt{|\det(g)|} dx_1 \wedge dx_2 \wedge dx_3 \wedge dx_4$.
Then, the bilinear form $v_x$ is given by
\begin{eqnarray*}
v_x: \Lambda^2 T_xM \times \Lambda^2 T_xM &\to & \R\\
(u,h) & \mapsto &\sqrt{|\det(g)|} (u^{12}h^{34}-u^{13}h^{24}+u^{14}h^{23}+\\
&&+u^{23} h^{14}- u^{24}h^{13}+u^{34}h^{12}).
\end{eqnarray*}
The associated quadratic form $v_{x}: \Lambda^2 T_x M \rightarrow \R$ is now given by 
\begin{equation}\label{eq: quadratic form for volume form}
v_x(u)=2\sqrt{|\det(g)|} (u^{12}u^{34}-u^{13}u^{24}+u^{14}u^{23}).
\end{equation}

\centerline{\textbf{The quadratic form $\Lambda^2g_x$ }:}
We need at this point the notion of the Kulkarni-Nomizu product. This product is defined for two symmetric $(2,0)$-tensors and gives as a result a $(4,0)$-tensor. Specifically, if $k$ and $l$ are symmetric $(2,0)$-tensors, then the product is defined by
\begin{eqnarray*}
(k \varowedge l) (u_1,u_2,u_3,u_4)&:=&k(u_1,u_3)l(u_2,u_4)+k(u_2,u_4)l(u_1,u_3)\\
&&-k(u_1,u_4)l(u_2,u_3)-k(u_2,u_3)l(u_1,u_4).
\end{eqnarray*}

Consider now the Riemannian metric $g_x$ and let $u=u_1 \wedge u_2$ and $h=h_1 \wedge h_2$. We define a symmetric bilinear form $\Lambda^2 g_x$ on $\Lambda^2 T_x M$ by defining it on totally decomposable vectors as  follows
\begin{eqnarray*}
\Lambda^2 g_x : \Lambda^2 T_x M \times \Lambda^2 T_x M &\to & \mathbb{R}\\
(u,h)  &\mapsto & \frac{1}{2} (g_x \varowedge g_x)(u_1,u_2,h_1,h_2)\\
& &=g_x(u_1,h_1)g_x(u_2,h_2)-g_x(u_1,h_2)g_x(u_2,h_1).
\end{eqnarray*}
and extending it bilinearly to a bilinear form on the whole $\Lambda^2 T_x M$. 

For $u$ and $h$ like in (\ref{eq: u}) and (\ref{eq: h}) we obtain, that  in components
\begin{equation*}
\Lambda^2 g_x (\frac{\partial}{\partial x^i} \wedge \frac{\partial}{\partial x^j} , \frac{\partial}{\partial x^k} \wedge \frac{\partial}{\partial x^l})=\det 
\begin{bmatrix}
g_{ik} & g_{jk}\\
g_{il} & g_{jl}
\end{bmatrix}
=\frac{1}{2}(g_x \varowedge g_x)_{ijkl}.
\end{equation*}
So we obtain a quadratic form
\begin{equation}\label{eq: quadratic form for Lambda^2 metric}
\Lambda^2 g_x(u)=\frac{1}{2} \underset{1 \leq i,k <j,l \leq 4}\sum  (g_x \varowedge g_x)_{ijkl} u^{ij} u^{kl}.
\end{equation}

\centerline{\textbf{The quadratic form $\mbox{R}_x$ }:}
Let now $\mbox{Rm}_x$ denote the $(4,0)$-Riemann curvature tensor at $x\in M$. We define a symmetric bilinear form $\mbox{R}_x$ on $\Lambda^2 T_x M$ by defining it on totally decomposable vectors as  follows
\begin{eqnarray*}
\mbox{R}_x: \Lambda^2 T_x M \times \Lambda^2 T_x M &\rightarrow  & \R \\
(u_1 \wedge u_2,h_1 \wedge h_2)& \mapsto & \mbox{Rm}_x(u_1,u_2,h_2,h_1).
\end{eqnarray*}
and extending it bilinearly to a bilinear form on the whole $\Lambda^2 T_x M$. 

In the basis  $\{\frac{\partial}{\partial x^i} \wedge \frac{\partial}{\partial x^j} \}_{1\leq i <j \leq 4}$ we obtain,
\begin{eqnarray*}
\mbox{R}_x: \Lambda^2 T_x M \times \Lambda^2 T_x M &\rightarrow  & \R \\
(\frac{\partial}{\partial x^i} \wedge \frac{\partial}{\partial x^j},\frac{\partial}{\partial x^k} \wedge \frac{\partial}{\partial x^l}) & \mapsto & 
R_{(ij)(kl)}=R_{ijlk},
\end{eqnarray*}
where $R_{ijlk}=\mbox{Rm}(\frac{\partial}{\partial x^i},\frac{\partial}{\partial x^j},\frac{\partial}{\partial x^l},\frac{\partial}{\partial x^k})$. Notice the convention $R_{(ij)(kl)}=R_{ijlk}$, which is used in the whole article.
The associated quadratic form is given by
\begin{equation}\label{quadratic form for curvature}
\mbox{R}_x(u)=\underset{1 \leq i,k <j,l \leq 4}\sum R_{ijlk} u^{ij} u^{kl}.
\end{equation}

From now on vector spaces are turned into complexified ones. The quadratic forms (\ref{eq: quadratic form for volume form}), (\ref{eq: quadratic form for Lambda^2 metric}) and (\ref{quadratic form for curvature}) define three quadrics in $\mathbb{P}(\Lambda^2 T_x M \otimes \C) \cong \mathbb{P}^5$, given by
\begin{equation}\label{eq: quadric of volume form}
\mathbb{P}(v_x)=\{[u] \in \mathbb{P}(\Lambda^2 T_x M \otimes \C): \,  u^{12}u^{34}-u^{13}u^{24}+u^{14}u^{23}=0 \},
\end{equation}
\begin{equation}\label{eq: quadric of Lambda^2 of metric}
\mathbb{P}(\Lambda^2 g_x)=\{[u] \in \mathbb{P}(\Lambda^2 T_x M \otimes \C): \, \frac{1}{2} \underset{1 \leq i,k <j,l \leq 4}\sum(g_x \varowedge g_x)_{ijkl} u^{ij} u^{kl} =0 \}
\end{equation}
and
\begin{equation}\label{eq: quadric of curvature}
\mathbb{P}(\mbox{R}_x)=\{[u] \in \mathbb{P}(\Lambda^2 T_x M \otimes \C): \, \underset{1 \leq i,k <j,l \leq 4}\sum R_{ijlk} u^{ij} u^{kl}=0 \}.
\end{equation}

We would like to take now a closer look at the Grassmannian $\mbox{Gr}_2(T_x M \otimes \C)$ of two-dimensional linear subspaces of $T_x M \otimes \C$. We prefer to look at it as the variety $\mbox{Gr}_1( \mathbb{P}(T_x M \otimes \C))$ of lines in $\mathbb{P}(T_x M \otimes \C)$, where $\mathbb{P}(T_x M \otimes \C) \cong \mathbb{P}^3$ . Let 
$$ w= \sum_{i=1}^{4} w^{i}\frac{\partial}{\partial x^i}$$
and
$$ \tilde{w}= \sum_{j=1}^{4} \tilde{w}^{j}\frac{\partial}{\partial x^j},$$
where $w,\tilde{w} \in T_x M \otimes \C$. Then $[w]=[w_{1},w_{2},w_{3},w_{4}]$ and $[\tilde{w}]=[\tilde{w}_{1},\tilde{w}_{2},\tilde{w}_{3},\tilde{w}_{4}]$
correspond to points in $\mathbb{P}(T_x M \otimes \C)$. Their projective span, denoted by $\mathbb{P}\mbox{-span}([w],[\tilde{w}])$ represents a line in $\mathbb{P}(T_x M \otimes \C)$.

Let $\mbox{pl}$ denote the Pl\" ucker embedding
\begin{eqnarray*}\label{eq: Pluecker embedding}
\mbox{pl}:\mbox{Gr}_1( \mathbb{P}(T_x M \otimes \C)) & \to & \mathbb{P}(\Lambda^2 (T_x M \otimes \C))\\
\mathbb{P}\mbox{-span}([w], [\tilde{w}]) & \mapsto & [w \wedge \tilde{w}]. \nonumber
\end{eqnarray*}
In other words, the Pl\"ucker embedding maps a line in $\mathbb{P}(T_x M \otimes \C)$ to a point in $\mathbb{P}(\Lambda^2 (T_x M \otimes \C))$. The coordinates of $[w \wedge \tilde{w}]$ in the basis $\{\frac{\partial}{\partial x^i} \wedge \frac{\partial}{\partial x^j} \}_{1\leq i <j \leq 4}$  are given by
\begin{equation*}
[w^1\tilde{w}^2- w^2 \tilde{w}^1, w^1 \tilde{w}^3 - w^3 \tilde{w}^1, w^1 \tilde{w}^4-w^4 \tilde{w}^1, w^2 \tilde{w}^3- w^3 \tilde{w}^2, w^2 \tilde{w}^4- w^4 \tilde{w}^2, w^3 \tilde{w}^4- w^4 \tilde{w}^3].
\end{equation*}
We will denote these coordinates by $[u^{12},u^{13},u^{14},u^{23},u^{24},u^{34}]$. Oberve that they correspond to the $2 \times 2$ minors of the matrix 
$$\begin{bmatrix}
w^1&&\tilde{w}^1\\
w^2&&\tilde{w}^2\\
w^3&&\tilde{w}^3\\
w^4&&\tilde{w}^4
\end{bmatrix}.$$

It is a well known fact, that $\mbox{Gr}_1(\mathbb{P}(T_x M \otimes \C))$ can be naturally realized as a quadric hypersurface in $\mathbb{P}(\Lambda^2 (T_x M \otimes \C))$. Recall that a vector $u \in \Lambda^2(T_x M \otimes \C)$ is called totally decomposable if there exist linear independent vectors $w, \tilde{w} \in T_x M \otimes \C$, such that $u=w \wedge \tilde{w}$. Observe that  
$$\mbox{pl}(\mbox{Gr}_1(\mathbb{P}(T_x M \otimes \C))=\{ [u] \in \mathbb{P}(\Lambda^2 (T_x M \otimes \C)): \, u \in \Lambda^2(T_x M \otimes \C) \mbox{ is totally decomposable} \}.$$ Furthermore, the vector $u \in  \Lambda^2(T_x M \otimes \C)$ is totally decomposable if and only if $u \wedge u=0$, in coordinates
$$u^{12}u^{34}-u^{13}u^{24}+u^{14}u^{23}=0.$$ 
By taking now into account (\ref{eq: quadric of volume form}), we observe that we can identify $\mbox{pl}(\mbox{Gr}_1( \mathbb{P}(T_x M \otimes \C)))$ with the quadric $\mathbb{P}(v_x)$.

\centerline{\textbf{The quadric surface $\mathbb{P}(g_x)$ }:}
Now the metric $g_x$ defines a quadratic form $T_x M \otimes \C  \rightarrow \C$ by
$$g_x(w)=\sum_{i,j=1}^{4} g_{ij} w^i w^j,$$
where $g_{ij}=g_{ji}$. It defines a quadric surface
$$\mathbb{P}(g_x)=\{ [w] \in \mathbb{P}(T_x M \otimes \C):\,  \sum_{i,j=1}^{4} g_{ij} w^i w^j=0 \}.$$ 
This quadric is non-degenerate, since the quadratic form $g_x$ is non-degenerate. So its rank equals four and it corresponds to a smooth quadric in $\mathbb{P}(T_x M \otimes \C)$. 

\begin{rem}
Recall, that if a quadric is mapped to a quadric under a projective trasformation, then the rank of the coefficient matrix is not changed. Thus one can classify quadrics in complex projective spaces up to their rank. Precisely, in $\mathbb{P}^3$ there are four of them: rank $4$ corresponds to a smooth quadric, rank $3$ to a quadric cone, rank $2$ to a pair of planes and rank $1$ to a double plane. The interested reader can look up page 33 of \cite{Harris}.
\end{rem} 

We need at this point some theory on spinor bundles. We will recall some facts on $spin$ and $spin^{\C}$ structures on $4$-manifolds. Heuristically, one can see $spin$ and $spin^{\C}$ structures as generalizations of orientantions. The tangent bundle $TM$ gives rise to a principal $O(4)$-bundle of frames denoted by $P_{O(4)}$. The manifold is said to be orientable if this bundle can be reduced to a $SO(4)$-bundle denoted by  $P_{SO(4)}$. We define the group $Spin(4)=SU(2)\times SU(2)$ to be the double cover of $SO(4)$. This is the universal cover. If we make a further reduction, we obtain a principal $Spin(4)$-bundle denoted by $P_{Spin(4)}$. We have then, that the map
$$\xi: P_{Spin(4)} \to P_{SO(4)}$$
is a double covering and say that the manifold is $spin$. To find the complex analogue we replace $SO(4)$ by the group $SO(4)\times \mathbb{S}^1$ and consider its double cover. We define the group 
$$Spin^{\C}(4)=(Spin(4) \times \mathbb{S}^1)/\{\pm 1\}=Spin(4) \times_{\mathbb{Z}_2} \mathbb{S}^1.$$
This is the desired double cover of $SO(4)\times \mathbb{S}^1$. Finally we define $M$ to be $spin^{\C}$, if given the bundle  $P_{SO(4)}$, there are principal bundles $P_{\mathbb{S}^1}$ and $P_{Spin^{\C}(4)}$, with a $Spin^{\C}(4)$ equivariant bundle map, a double cover
$$\xi^{\prime}:P_{Spin^{\C}(4)} \to P_{SO(4)} \times P_{\mathbb{S}^1}.$$
It is a known fact, that in dimension four any orientable manifold has a (non-unique) $spin^{\C}$ structure. The $spin^{\C}$ representation now allows us to consider the associated vector bundle $S$, called the spinor bundle for a given $spin^{\C}$ structure. This is a complex vector bundle. In the four-dimensional case this vector bundle splits into the sum of two subbundles $S^+$, $S^-$, such that
$$S=S^+ \oplus S^-.$$

Let $\mathbb{P}(S^{+}_{x})\cong \mathbb{P}^1$ and $\mathbb{P}(S^{-}_{x})\cong \mathbb{P}^1$ denote the projectivizations of the fibers of the spinor bundles $S^+$  and $S^-$ over $x$ respectively. Consider now the Segre embedding
\begin{eqnarray*}
\mathbb{P}(S^{-}_{x}) \times \mathbb{P}(S^{+}_{x}) &\to & \mathbb{P}(S^{-}_{x} \otimes S^{+}_{x})\\
\left[ \rho^{-} \right] \times\left[ \rho^{+} \right] &\mapsto &\left[ \rho^{-} \otimes \rho^{+} \right].
\end{eqnarray*}
One can show, that $S^{-}_{x} \otimes S^{+}_{x}  \cong T_x M \otimes \C$.

Let now $\{ e_i \}_{i=1}^{4}$ be a local orthonormal frame for $T_xM \otimes \C$. We will be working with this frame from now on, because it is more convient for computational reasons. The Segre embedding with respect to the basis $\{ e_i \}_{i=1}^{4}$ is given by
\begin{eqnarray}\label{eq: segre}
\sigma:\mathbb{P}(S^{-}_{x}) \times \mathbb{P}(S^{+}_{x}) &\rightarrow & \mathbb{P}(T_x M \otimes \C) \nonumber\\ 
([a^1,a^2],[b^1,b^2]) &\mapsto & [a^1b^1+a^2b^2,i(a^2b^2-a^1b^1),-i(a^1b^2+a^2b^1),a^2b^1-a^1b^2]  \nonumber\\
&&=:[w^1,w^2,w^3,w^4].  
\end{eqnarray}
This is a well defined map. In order to pick coordinates on $\mathbb{P}(S^{-}_{x}) $ and $\mathbb{P}(S^{+}_{x}) $ one should observe the projection of $\xi^{\prime}$ onto the first factor:
$$P_{Spin^{\C}(4)} \to P_{SO(4)}.$$
A point in the fiber of $P_{SO(4)}$ over $x$ is a basis for $T_x M$ and a point in the fiber of $P_{Spin^{\C}(4)}$ over $x$ is a basis for the spinor $S_x= S^{+}_{x} \oplus S^{-}_{x}$.

\begin{rem}
Recall that the "classical" Segre embedding is given by
\begin{eqnarray*}
\Sigma:\mathbb{P}^{1} \times \mathbb{P}^{1} &\rightarrow & \mathbb{P}^3 \\
([a^1,a^2],[b^1,b^2]) &\mapsto & [a^1 b^1,a^2 b^2, a^1 b^2, a^2 b^1]=:[W^1,W^2,W^3,W^4].
\end{eqnarray*}
The image is just the quadric surface $W^1 W^2 - W^3 W^4 =0$ and the rank of the quadric is four, i.e. it's a smooth quadric.
The associated symmetric matrix is 
$$\Sigma=\begin{bmatrix}
0&1/2&0&0\\
1/2&0&0&0\\
0&0&0&-1/2\\
0&0&-1/2&0
\end{bmatrix}.$$
Let now 
$$B=\begin{bmatrix}
1&i&0&0\\
1&-i&0&0\\
0&0&i&-1\\
0&0&i&1
\end{bmatrix},$$
so that $B^t \Sigma B =I_4$. Then 
$$B^{-1}\begin{bmatrix}
W^1\\W^2\\W^3\\W^4
\end{bmatrix}=\begin{bmatrix}
w^1\\w^2\\w^3\\w^4
\end{bmatrix}.$$
\end{rem}

One we can easily observe that the image of the Segre embedding is just the quadric surface $(w^{1})^{2}+(w^{2})^{2}+(w^{3})^{2}+(w^{4})^{2}=0$ and the rank of the quadric is four, i.e. it s a smooth quadric. Thus $\mathbb{P}(g_x)$ can be written with respect to the orthonormal basis $\{e_i\}_{i=1}^{4}$ for $T_x M \otimes \C$ as
\begin{equation*} \label{eq: quadric of the metric}
\mathbb{P}(g_x)=\{[w] \in \mathbb{P}(T_x M \otimes \C):\, (w^{1})^{2}+(w^{2})^{2}+(w^{3})^{2}+(w^{4})^{2}=0 \}.
\end{equation*}

The quadric $\mathbb{P}(g_x)$ has two rulings by lines and a unique line of each ruling passes through each point of the quadric. More precisely:  fix a point $[a^1,a^2] \in \mathbb{P}(S_{x}^{-})$. Then 
$$t_+:=\sigma(\{ [a^1,a^2] \} \times \mathbb{P}(S^{+}_{x}))$$
is a line in $\mathbb{P}(T_x M \otimes \C)$. Similarly for fixed $[b^1,b^2] \in \mathbb{P}(S_{x}^{+})$, 
$$t_-:=\sigma(\mathbb{P}(S^{-}_{x}) \times \{[b^1,b^2]\})$$
is also  a line in $\mathbb{P}(T_x M \otimes \C)$.  So the quadric contains two families of lines denoted by  $\mathcal{F}_-$ and $\mathcal{F}_+$ respectively such that,
$$\mathcal{F}_-=\bigcup_{[b^1,b^2] \in \mathbb{P}(S^{+}_{x})}\{t_-\} ,\,\,\, \mathcal{F}_+=\bigcup_{[a^1,a^2] \in \mathbb{P}(S^{-}_{x}) } \{t_+\}.$$
If we choose any point of $t_+ $, we can find a unique line of the family $\mathcal{F}_-$ passing through it. Analogously for every point of $t_{-}$ we can find a unique line of $\mathcal{F}_+$ passing through it. Furthermore it holds that no two lines from the same family intersect and that any two lines belonging to different families intersect in a unique point of the quadric. 
The lines $\mathbb{P}(S_{x}^{\pm})$ are called the rectilinear generators of the quadric and 

\begin{equation*}\label{eq: P(g) generators}
\mathbb{P}(g_x) = \sigma( \mathbb{P}(S^{-}_{x})  \times  \mathbb{P}(S^{+}_{x})).
\end{equation*}

\centerline{\textbf{The Pl\"ucker Embedding}:}
Every $t_-$ or  $t_+$ is a line in $\mathbb{P}(T_x M \otimes \C)$. We will compute their images under the Pl\"ucker embedding. By setting first $[b^1,b^2]=[1,0]$ and then $ [b^1,b^2]=[0,1]$  in (\ref{eq: segre}) we can easily see, that
$$t_+=\mathbb{P}\mbox{-span}([a^1,-ia^1,-ia^2,a^2],[a^2,ia^2,-ia^1,-a^1]).$$
Thus we obtain, that the coordinates of $\mbox{pl}(t_+)$ in the basis $\{e_i \wedge e_j \}_{1\leq i<j\leq 4}$ of $\Lambda^2(T_x M \otimes \C)$ are
\begin{equation*}
[2ia^1 a^2,i\{(a^{2})^{2}-(a^{1})^{2}\},-(a^{1})^{2}-(a^{2})^{2},-(a^{1})^{2}-(a^{2})^{2},i\{(a^{1})^{2}-(a^{2})^{2}\},2ia^1 a^2].
\end{equation*}
On the other hand by setting first $[a^1,a^2]=[1,0]$ and then $ [a^1,a^2]=[0,1]$  in (\ref{eq: segre}), we have that
$$t_-=\mathbb{P}\mbox{-span}([b^1,-ib^1,-ib^2,-b^2],[b^2,ib^2,-ib^1,b^1]).$$
In this case the coordinates of $\mbox{pl}(t_-)$ in the basis $\{e_i \wedge e_j \}_{1\leq i<j\leq 4}$ of $\Lambda^2(T_x M \otimes \C)$ are
\begin{equation*}
[2ib^1 b^2,i\{(b^{2})^{2}-(b^{1})^{2}\},(b^{1})^{2}+(b^{2})^{2},-(b^{1})^{2}-(b^{2})^{2},i\{(b^{2})^{2}-(b^{1})^{2}\},-2ib^1 b^2].
\end{equation*}

It is well known, that in dimension four the Hodge $\ast$-operator induces a natural decomposition of $\Lambda^2 T_xM$ on an oriented manifold $M$ given by
$$\Lambda^2 T_xM  =\Lambda^{2}_{+} T_xM  \oplus \Lambda^{2}_{-} T_xM,$$
where $\Lambda^{2}_{+} T_xM$  and $\Lambda^{2}_{-} T_xM$ correspond to the eigenspaces $+1$ and $-1$ respectively. Furthermore elements of $\Lambda^{2}_{+} T_xM$  and $\Lambda^{2}_{-} T_xM$ are called self-dual and anti-self-dual respectively. We will perform a change of basis for $\Lambda^2 T_x M \otimes \C$. We would like to express the coordinates of $\pl(t_+)$ and $\pl(t_-)$ in the basis $\mathfrak{B}$ of $\Lambda^2(T_x M \otimes \C)$ given by
\begin{eqnarray*}
f_{1}^{\pm}=\frac{1}{\sqrt{2}}(e_1 \wedge e_2 \pm e_3 \wedge e_4)\\
f_{2}^{\pm}=\frac{1}{\sqrt{2}}(e_1 \wedge e_3 \mp e_2 \wedge e_4)\\
f_{3}^{\pm}=\frac{1}{\sqrt{2}}(e_1 \wedge e_4 \pm e_2 \wedge e_3).
\end{eqnarray*}
One can observe, that $\{f_{i}^{+}\}_{i=1}^{3}$ is a basis for $\Lambda^{2} _{+ }T_x M \otimes \C$, where $\ast f_{i}^{+}=f_i$, $i=1,2,3$
and that  $\{f_{i}^{-}\}_{i=1}^{3}$ is a basis for $\Lambda^{2} _{-}T_x M \otimes \C$, where $\ast f_{i}^{-}=-f_i$, $i=1,2,3$. By using the change of basis matrix
$$\begin{bmatrix}
\frac{\sqrt{2}}{2} & 0 &0 &0 &0 &\frac{\sqrt{2}}{2}  \\
0& \frac{\sqrt{2}}{2} & 0 &0 &- \frac{\sqrt{2}}{2} & 0 \\
0 & 0 & \frac{\sqrt{2}}{2}  &\frac{\sqrt{2}}{2} & 0 &0\\
\frac{\sqrt{2}}{2} &0&0 &0 &0& -\frac{\sqrt{2}}{2} \\
0& \frac{\sqrt{2}}{2} & 0 & 0 & \frac{\sqrt{2}}{2} & 0\\
0&0& \frac{\sqrt{2}}{2} & -\frac{\sqrt{2}}{2} &0&0
\end{bmatrix}$$
we compute, that the coordinates $[u^{12},u^{13},u^{14},u^{23},u^{24},u^{34}]$ in the basis $\mathfrak{B}$ of $\Lambda^2(T_x M \otimes \C)$ are given by
$$[u^1,u^2,u^3,u^4,u^5,u^6]:=[u^{12}+u^{34},u^{13}-u^{24},u^{14}+u^{23},u^{12}-u^{34},u^{13}+u^{24},u^{14}-u^{23}].$$
Thus the coordinates of $\mbox{pl}(t_+)$ in the basis $\mathfrak{B}$ of $\Lambda^2(T_x M \otimes \C)$ are
\begin{equation}\label{eq: homogeneous coordinates of t_+ in P^5}
[2 ia^1 a^2,i\{(a^{2})^{2}-(a^{1})^{2}\},-(a^{1})^{2}-(a^{2})^{2},0,0,0]
\end{equation}
and the coordinates of $\mbox{pl}(t_-)$ in the basis $\mathfrak{B}$ of $\Lambda^2(T_x M \otimes \C)$ are
\begin{equation}\label{eq: homogeneous coordinates of t_ - in P^5}
[0,0,0,2 ib^1 b^2,i\{(b^{2})^{2}-(b^{1})^{2}\},(b^{1})^{2}+(b^{2})^{2}].
\end{equation}

By (\ref{eq: homogeneous coordinates of t_+ in P^5}) and (\ref{eq: homogeneous coordinates of t_ - in P^5}) one can easily see, that $\mathcal{F}_+ $ and $\mathcal{F}_-$ are embedded conics in $\mathbb{P}(\Lambda^2 T_x M \otimes \C)$ given by the equations

\begin{equation}\label{eq:conic P(g_x)_+}
\begin{cases} u^{4}=u^{5}=u^{6}=0\\(u^{1})^{2}+(u^{2})^{2} +(u^{3})^{2}=0 \end{cases}
\end{equation}
and
\begin{equation}\label{eq:conic P(g_x_-)}
\begin{cases} u^{1}=u^{2}=u^{3}=0\\(u^{4})^2+(u^{5})^2+ (u^{6})^2=0 \end{cases}
\end{equation}
respectively. We will denote these conics by $\mathcal{C}_+$ and $\mathcal{C}_-$. Obviously each of the two conics is sitting in a plane in $\mathbb{P}(\Lambda^2 T_x M \otimes \C)$. The first plane is $\mathbb{P}(\Lambda^{2}_{+} T_x M \otimes \C)$ and the second is $\mathbb{P}(\Lambda^{2}_{-} T_x M \otimes \C)$. They are given by the equations

\begin{equation*}\label{eq:plane P_+}
u^{4}=u^{5}=u^{6}=0
\end{equation*}
and
\begin{equation*}\label{eq:plane P_-}
u^{1}=u^{2}=u^{3}=0
\end{equation*}
respectively. Obviously $\mathbb{P}(\Lambda^{2}_{+} T_x M \otimes \C) \cap \mathbb{P}(\Lambda^{2}_{-} T_x M \otimes \C)= \emptyset$. 

\centerline{\textbf{The projectivized tangent bundle}:}
Let now $\mathcal{T}:=T\mathbb{P}(g_x)$ denote the tangent bundle of the quadric $\mathbb{P}(g_x)$ and $\mathbb{P}(\mathcal{T})$ its projectivization. Then one can write
$$\mathbb{P}(\mathcal{T}) = \{ (t_+ \cap t_-,l):\, l \subset \mathbb{P}(T_x M \otimes \C)  \mbox{ is a line } \mbox{tangent to } \mathbb{P}(g_x) \mbox{ at the point } t_+ \cap t_-\} ,$$
which is an algebraic subvariety of $\mathbb{P}(g_x) \times \mbox{Gr}_1(\mathbb{P}(T_x M \otimes \C)) \subset \mathbb{P}(T_x M \otimes \C)   \times  \mbox{Gr}_1(\mathbb{P}(T_x M \otimes \C))$.
We will now  apply the Pl\"ucker embedding on the second factor. We define the map
$$\mbox{id}_{\mathbb{P}(g_x)} \times \mbox{pl}: \mathbb{P}(g_x) \times \mbox{Gr}_1 (\mathbb{P}(T_x M \otimes \C)) \to \mathbb{P}(g_x) \times \mbox{pl}(\mbox{Gr}_1 (\mathbb{P}(T_x M \otimes \C))). $$ 
Then
\begin{eqnarray*}
(\mbox{id}_{\mathbb{P}(g_x)} \times \mbox{pl})  \big( \mathbb{P}(\mathcal{T}) \big) &:=& \{ (t_+ \cap t_-,\pl(l)):\, l \subset \mathbb{P}(T_x M \otimes \C) \\
&&  \mbox{ is a line } \mbox{tangent to } \mathbb{P}(g_x) \mbox{ at the point } t_+ \cap t_-\}.
\end{eqnarray*}
Thus $(\mbox{id}_{\mathbb{P}(g_x)} \times \mbox{pl})  \big( \mathbb{P}(\mathcal{T}) \big)$ is naturally an algebraic subvariety
\begin{equation*}
(\mbox{id}_{\mathbb{P}(g_x)} \times \mbox{pl})  \big( \mathbb{P}(\mathcal{T}) \big) \subset \mathbb{P}(g_x) \times \mbox{pl}(\mbox{Gr}_1(\mathbb{P}(T_x M \otimes \C))) \subset \mathbb{P}(T_x M \otimes \C) \times \mathbb{P}(\Lambda^2 T_x M \otimes \C).
\end{equation*}
If we now denote by 
\begin{eqnarray*}
\pi: (\mbox{id}_{\mathbb{P}(g_x)} \times \mbox{pl})  \big( \mathbb{P}(\mathcal{T}) \big) &\to& \,\mbox{pl}(\mbox{G}r_1(\mathbb{P}(T_x M \otimes \C))) \\
(t_+ \cap t_- , \pl(l)) &\mapsto & \pl(l)
\end{eqnarray*}
and
\begin{eqnarray*}
\tau: (\mbox{id}_{\mathbb{P}(g_x)} \times \mbox{pl})  \big( \mathbb{P}(\mathcal{T}) \big) &\to& \mathbb{P}(g_x)  \\
(t_+ \cap t_- , \pl(l)) &\mapsto & t_+ \cap t_-
\end{eqnarray*}
the natural projections, we are interested in the geometry of $\pi \Big((\mbox{id}_{\mathbb{P}(g_x)} \times \mbox{pl})  \big( \mathbb{P}(\mathcal{T}) \big) \Big)$. We would like to give a description in $\mathbb{P}(\Lambda^2 T_x M \otimes \C)$ of the image of the set of lines tangent to the quadric $\mathbb{P}(g_x)$ at the point $t_+ \cap t_-$ under the Pl\"ucker embedding. All these lines lie on one plane and pass through one point, so in $\mathbb{P}(\Lambda^2 T_x M \otimes \C)$ they form a line given by  $\mathbb{P}\mbox{-span}(\pl(t_+),\pl(t_-))$. Thus
\begin{equation}\label{eq:pi(P(T))}
\pi \Big( (\mbox{id}_{\mathbb{P}(g_x)} \times \mbox{pl})  \big( \mathbb{P}(\mathcal{T}) \big)\Big)= \{ \pl(l) \in \mathbb{P}\mbox{-span}(\pl(t_+),\pl(t_-)): \,\pl(t_+) \in \mathcal{C}_+, \, \pl(t_-) \in\mathcal{C}_- \}
\end{equation}
and 
$$\mbox{dim}\Big[ \pi \Big( (\mbox{id}_{\mathbb{P}(g_x)} \times \mbox{pl})  \big( \mathbb{P}(\mathcal{T}) \big)\Big) \Big]=\mbox{dim}(\mathcal{C}_+)+\mbox{dim}(\mathcal{C}_-)+1=3,$$ 
because $\pi \Big( (\mbox{id}_{\mathbb{P}(g_x)} \times \mbox{pl})  \big( \mathbb{P}(\mathcal{T}) \big)\Big)$ is the join of the varieties $\mathcal{C}_+$ and $\mathcal{C}_-$. We can now describe $(\mbox{id}_{\mathbb{P}(g_x)} \times \mbox{pl})  \big( \mathbb{P}(\mathcal{T}) \big)$ by

\begin{equation*}\label{eq: P(T)}
(\mbox{id}_{\mathbb{P}(g_x)} \times \mbox{pl})  \big( \mathbb{P}(\mathcal{T}) \big)= \{( t_+ \cap t_-, \pl(l)) : \,  \pl(l) \in \mathbb{P}\mbox{-span}(\pl(t_+),\pl(t_-)) \}.
\end{equation*}

We are going to show now that the variety $\pi \Big( (\mbox{id}_{\mathbb{P}(g_x)} \times \mbox{pl})  \big( \mathbb{P}(\mathcal{T}) \big) \Big)$ is singular and we will determine its singular locus. By (\ref{eq:conic P(g_x)_+}), (\ref{eq:conic P(g_x_-)}) and (\ref{eq:pi(P(T))}) we see that the variety $\pi \Big( (\mbox{id}_{\mathbb{P}(g_x)} \times \mbox{pl})  \big( \mathbb{P}(\mathcal{T}) \big) \Big)$ is defined by the equations
\begin{equation}\label{eq:tangents to the quadric eta}
\begin{cases} (u^{1})^{2}+(u^{2})^{2} +(u^{3})^{2}=0\\ (u^{4})^{2}+(u^{5})^{2} +(u^{6})^{2}=0. \end{cases}
\end{equation}

The system of equations (\ref{eq:tangents to the quadric eta}) shows that the singular points of $\pi \Big( (\mbox{id}_{\mathbb{P}(g_x)} \times \mbox{pl})  \big( \mathbb{P}(\mathcal{T}) \big) \Big)$ are given by 
$$\mbox{Sing} \Big( \pi \Big( (\mbox{id}_{\mathbb{P}(g_x)} \times \mbox{pl})  \big( \mathbb{P}(\mathcal{T}) \big) \Big) \Big)=\mathcal{C}_+ \cup \mathcal{C}_-.$$
Let's explain why. We will fix a coordinate system on $\mathbb{P}\mbox{-span}(\pl(t_+),\pl(t_-))$. Let $T_+$ and $T_-$ denote the vector space representations of $\pl(t_+)$ and $\pl(t_-)$ in the basis $\mathcal{B}$ of $\Lambda^2(T_x M \otimes \C)$ respectively. We have then, that
$$\mbox{span}(T_+,T_-)=\{ \lambda T_+ + \mu T_-: \, \lambda,\mu \in \C \}$$
is a plane in $\Lambda^2 T_x M \otimes \C$. So we obtain a projective coordinate system on $\mathbb{P}\mbox{-span}(\pl(t_+),\pl(t_-))$. A point on this line has coordinates in the basis $\mathfrak{B}$ of $\Lambda^2(T_x M \otimes \C)$ given by
\begin{equation*}\label{eq: point on <t_+, t_->}
[2 \lambda  ia^1a^2,\lambda i\{(a^{2})^{2}-(a^{1})^{2}\},\lambda \{-(a^{1})^{2}-(a^{2})^{2}\},2\mu i b^1 b^2,\mu i\{(b^{2})^{2}-(b^{1})^{2}\},\mu\{(b^{1})^{2}+(b^{2})^{2}\}]
\end{equation*}
for scalars $\lambda$ and $\mu$. Obviously, by  (\ref{eq:tangents to the quadric eta}) the Jacobian matrix of the polynomials defining the variety is 

$$\begin{bmatrix}
2u^{1} & 2u^{2} & 2u^{3} & 0 & 0 &0 \\
0 & 0 & 0 & 2u^{4} & 2u^{5} & 2u^{6} \\
\end{bmatrix}$$
and its rank at the point $\pl(t_+)$ or $\pl(t_-)$ is equal to one, i.e. lower than on any other point of  $\mathbb{P}\mbox{-span}(\pl(t_+),\pl(t_-))$.

\subsection{The intersection of three quadrics}
Consider the intersection 
\begin{eqnarray*}
S_x &=& \mathbb{P}(v_x)\cap \mathbb{P}(\Lambda^2 g_x) \cap \mathbb{P}(\mbox{R}_x)\\
&=&\mbox{pl}(\mbox{Gr}_1(\mathbb{P}( T_x M \otimes \C))) \cap \mathbb{P}(\Lambda^2 g_x) \cap \mathbb{P}(\mbox{R}_x)
\end{eqnarray*}
of the three quadrics  in $\mathbb{P}(\Lambda^2 T_x M \otimes \C)$.  We consider a line $l$ tangent to the quadric $\mathbb{P}(g_x)$. By the discussion in the previous subsection  it corresponds to a point in $\mbox{pl}(\mbox{Gr}_1(\mathbb{P}( T_x M \otimes \C)))$. The condition that the line $l$ is tagent to the quadric $\mathbb{P}(g_x)$ is equivalent to the condition that $\pl(l) \in \mathbb{P}(\Lambda^2 g_x)$. So 
$$\pi \Big( (\mbox{id}_{\mathbb{P}(g_x)} \times \mbox{pl})  \big( \mathbb{P}(\mathcal{T}) \big) \Big)=\mbox{pl}(\mbox{Gr}_1(\mathbb{P}( T_x M \otimes \C))) \cap \mathbb{P}(\Lambda^2 g_x).$$
This means that,
$$S_x=\pi \Big( (\mbox{id}_{\mathbb{P}(g_x)} \times \mbox{pl})  \big( \mathbb{P}(\mathcal{T}) \big) \Big) \cap \mathbb{P}(\mbox{R}_x).$$
Therefore $S_x$ must have singularities 
$$\mbox{Sing}(S_x) \supset \mbox{Sing}\Big( \pi \Big( (\mbox{id}_{\mathbb{P}(g_x)} \times \mbox{pl})  \big( \mathbb{P}(\mathcal{T}) \big) \Big) \Big) \cap \mathbb{P}(\mbox{R}_x)=(\mathcal{C}_+  \cap \mathbb{P}(\mbox{R}_x)) \cup (\mathcal{C}_-  \cap \mathbb{P}(\mbox{R}_x)) .$$

\begin{defi}
The variety $S_x$ is called the local invariant of the Riemannian manifold $(M,g)$ at the point $x$.
\end{defi}

\begin{rem}
Notice, that if $\mbox{R}_x= \kappa \Lambda^2 g_x$, $\kappa \in \C^*$, the manifold at the point $x$ is a manifold of constant curvature in any two dimensional direction. In such a case, $S_x$ is not defined and we shall not consider such points on $M$.
\end{rem}

In the folowing we assume that the quadric $\mathbb{P}(\mbox{R}_x)$ intersects the non-singular points of $\pi \Big( (\mbox{id}_{\mathbb{P}(g_x)} \times \mbox{pl})  \big( \mathbb{P}(\mathcal{T}) \big) \Big)$ transversally and intersects the singular locus $\mathcal{C}_+ \cup  \mathcal{C}_-$ transversally as well. It follows by \cite{Harris}, Proposition 17.18 that  $S_x$ is the complete intersection of the quadrics $\mathbb{P}(v_x)$, $\mathbb{P}(\Lambda^2 g_x)$, $\mathbb{P}(\mbox{R}_x)$. 

\begin{rem}
Recall that two varieties intersect transversally if they intersect transversally at each point of their intersection, i.e. they are smooth at this point and their separate tangent spaces at that point span the tangent space of the ambient variety at that point. In other words if $X$ and $Y$ are projective subvarieties of $\mathbb{P}^n$, then $X$ and $Y$ intersect transversally if at every point $u \in X \cap Y$, $T_{u} X \oplus T_{u} Y = T_{u} \mathbb{P}^n$. Thus transversality  depends on the choice of the ambient variety. In particular, transversality always fails whenever two subvarieties are tangent.
\end{rem}

Recall that the complete intersection of three quadrics in $\mathbb{P}^5$ is a $K3$ surface. An exposition on $K3$ surfaces can be found in the Appendix of \cite{Tergiakidis}. Thus $S_x$ is a (singular) $K3$ surface. The quadric $\mathbb{P}(\mbox{R}_x)$ interesects the singular locus $\mathcal{C}_+ \cup \mathcal{C}_-$ transversally and each intersection $\mathbb{P}(\mbox{R}_x) \cap \mathcal{C}_+$, $\mathbb{P}(\mbox{R}_x) \cap \mathcal{C}_-$, consists of four ordinary double points (the transversal intersection of quadric and conic gives a $0$-dimensional variety of degree $4$). We will denote the set of these points by  $$\mbox{Sing}(S_x)=\{\pl(t_{+}^{1}),\mbox{pl}(t_{+}^{2}) ,\mbox{pl}(t_{+}^{3}) ,\mbox{pl}(t_{+}^{4}) ,\mbox{pl}(t_{-}^{1}) ,\mbox{pl}(t_{-}^{2}) ,\mbox{pl}(t_{-}^{3}) ,\mbox{pl}(t_{-}^{4}) \}.$$

Consider now the the algebraic subvariety
$$\tilde{S}_x=\pi^{-1}(\mathbb{P}(\mbox{R}_x)) \subset \mathbb{P}(T_x M \otimes \C) \times \mathbb{P}(\Lambda^2 T_x M \otimes \C).$$
The next step is to show, that $\tilde{S}_x$ is the resolution of the singular points of $S_x$. We consider the map
$$\tilde{\pi}:\tilde{S}_x \to S_x.$$
Then $\tilde{S}_x$ is the resolution of the singular points of $S_x$, if and only if
$$\tilde{S}_x \setminus \tilde{\pi}^{-1}(\mbox{Sing}(S_x)) \cong  S_x \setminus \mbox{Sing}(S_x).$$
By the definiton of $\tilde{\pi}^{-1}$ this is indeed an isomorphism. 

We would like to compute now $\tilde{\pi}^{-1}(\mbox{Sing}(S_x))$, or in other words to find the blow ups of the singular points $\pl(t_{+}^{i})$, $\pl(t_{-}^{j})$, $1\leq i,j \leq 4$.

\begin{rem}
Let's recall the notion of the blow up of a complex surface at a point. Let  $q\in U \subset X$ be an open neighborhood and $(x,y)$ local coordinates such that $q=(0,0)$ in this coordinate system. Define
$$\tilde{U}:=\{((x,y),[z,w]) \in U \times \mathbb{P}^1: \, xw=yz\}.$$
We have then the projection onto the first factor
\begin{eqnarray*}
p_U: \tilde{U} & \to & U \\
((x,y),[z,w]) & \mapsto & (x,y).
\end{eqnarray*}
If $(x,y) \neq (0,0)$, then $p_{U}^{-1}((x,y))=(((x,y),[z,w]) )$.
Furthermore we have $p_{U}^{-1}(q)=\{q\} \times \mathbb{P}^1$.
This implies that the restriction
$$p_U: p_{U}^{-1}(U \setminus \{q\}) \to U \setminus \{q\}$$
is an isomorphism and $p_{U}^{-1}(q) \cong \mathbb{P}^1$ is a curve contracted by $p_U$ to a point. Now let us take the gluing of $X$ and $\tilde{U}$ along $X \setminus \{q\}$ and $\tilde{U} \setminus \{q\} \cong U  \setminus \{q\}$. In this way we obtain a surface $\tilde{X}$ together with a morphism $p: \tilde{X} \to X$. Notice that $p$ gives an isomoprhism between $X \setminus \{q\}$ and $\tilde{X} \setminus p^{-1}(q)$ and contracts the curve $\mathbb{P}^1 \cong p^{-1}(q)$ to the point $q$.
The morphism $p: \tilde{X} \to X$ is called the blow up of $X$ along $q$. The curve $p^{-1}(q)\cong \mathbb{P}^1$ is called exceptional curve or exceptional divisor of the blow-up.
\end{rem}

\begin{comment}
Consider the map
$$\tilde{\tau}:\tilde{S}_x \rightarrow \mathbb{P}(g_x) .$$
Let's denote with $s_+$ and $s_-$ sections of $\tilde{\tau}$, such that
$$s_{\pm}(t_+ \cap t_-) = \mathbb{P}(T_{t_+ \cap t_-} t_{\pm}) \in \mathbb{P} (T_{t_+ \cap t_-} \mathbb{P}(g_x)).$$

\begin{rem}
The fact that $\mathbb{P}(T_{t_+ \cap t_-} t_{\pm})$ is a point in $\mathbb{P} (T_{t_+ \cap t_-} \mathbb{P}(g_x))$ comes from the following easy observation. Let $t_{\pm} \subset \mathbb{P}(g_x)$. Then $T_{t_+ \cap t_-} t_{\pm} \subset T_{t_+ \cap t_-} \mathbb{P}(g_x)$ But $T_{t_+ \cap t_-} t_{\pm}$ is $1$-dimensional, so its projectivizations is just a point, i.e.  $\mathbb{P}(T_{t_+ \cap t_-} t_{\pm}) \in \mathbb{P} (T_{t_+ \cap t_-} \mathbb{P}(g_x))$. Furthermore one can deduce that the projectivized tangent bundle of $t_{\pm}$ is a line and is a subset of $\mathbb{P}(\mathcal{T})$.
\end{rem}
\end{comment}

We obtain that
\begin{equation*}\label{eq: blow-ups t_+}
E_i:=\tilde{\pi}^{-1}(\pl(t_{+}^{i}))=\{ (t_{+}^{i} \cap t_-,\pl(t_{+}^{i})):\, t_{-} \subset \mathcal{F}_{-}\}  \cong \mathbb{P}^1,
\end{equation*}
\begin{equation*}\label{eq: blow-ups t_-}
F_j:=\tilde{\pi}^{-1}(\pl(t_{-}^{j}))=\{ (t_+ \cap t_{-}^{j},\pl(t_{-}^{j})):\, t_{+} \subset \mathcal{F}_{+}\} \cong \mathbb{P}^1,
\end{equation*}
for $1 \leq i,j \leq 4$.
Observe that this means, that  $\tilde{\pi}$  is the blow-up of $S_x$ along $\pl(t_{+}^{i}), \pl(t_{-}^{j})$ for $1 \leq i,j \leq 4$ and the curves $E_i, F_j$, for $1 \leq i,j \leq 4$ are the exceptional divisor of the blow-up. In other words, $\tilde{\pi}$ contracts the curves $E_i$ to the points $\pl(t_{+}^{i})$ and the curves $F_j$ to the points $\pl(t_{-}^{j})$ for $1 \leq i,j \leq 4$.

\centerline{\textbf{The branching curve $\Gamma_x$ }:}
We will show that the map
$$\tilde{\tau}:\tilde{S}_x \rightarrow \mathbb{P}(g_x) $$
is a double branched cover at a generic point, where $\tilde{\tau}$ is the restriction of $\tau$ to $\tilde{S}_x$. The term "double branched cover" means, that there exists a closed subset $\mbox{Br}$ of $\mathbb{P}(g_x)$, such that $\tilde{\tau}$ restricted to $\tilde{S}_x \setminus \mbox{Ram}$, where $\mbox{Ram}:= \tilde{\tau}^{-1}(\mbox{Br})$, is a topological double cover of $\mathbb{P}(g_x) \setminus \mbox{Br}$. Points in $\mbox{Br}$ and $\mbox{Ram}$ are called branching points and ramification points respectively. The term "generic" stands for the fact that, sometimes $\tilde{\tau}$ represents a branched double cover followed by a blow up.  Before describing the preimage $\tilde{\tau}^{-1}(t_+ \cap t_-)$ we would like to be more precise.

The block decomposition of the Riemann curvature operator in dimension four is given by
$$\cop=\begin{bmatrix}
A && B\\
B^t && C
\end{bmatrix},$$
where $A$ and $C$ correspond to the operators associated to
$$W_+ + \frac{\mbox{scal}}{24} g \varowedge g$$ 
and 
$$W_- + \frac{\mbox{scal}}{24} g \varowedge g$$
respectively and $B$ is the operator associated to the curvature-like tensor  
$$ \frac{1}{2}\overset{\circ}{\mbox{Ric}} \varowedge g=\frac{1}{2} \Big( \mbox{Ric}-\frac{\mbox{scal}}{4}g \Big) \varowedge g.$$
Recall, that in dimension four
$$\mbox{Rm}=W_++W_- + \frac{1}{2}\overset{\circ}{\mbox{Ric}} \varowedge g+\frac{\mbox{scal}}{24} g \varowedge g,$$
where $W_+, W_-$ denote the Weyl parts of the curvature and $\overset{\circ}{\mbox{Ric}}$ the traceless Ricci tensor.

Consider now the block decomposition above and let $u=u_1 + u_2 \in \Lambda_{+}^{2} (T_xM \otimes \C) \oplus \Lambda_{-}^{2} (T_xM \otimes \C)$. Then
\begin{eqnarray*}
\mbox{R}_x(u) &=& \Lambda^2g_x(\mathfrak{Rm}(u),u) \\
&=& \Lambda^2 g_x(A(u_1),u_1)+\Lambda^2g_x(B (u_2),u_1)+ \Lambda^2g_x(B^t(u_1),u_2)+\Lambda^2g_x(C(u_2),u_2)\\
&=& \Lambda^2 g_x(A(u_1),u_1)+2\Lambda^2g_x(B (u_2),u_1)+\Lambda^2g_x(C(u_2),u_2).\\
\end{eqnarray*}
Now the quadric $\mathbb{P}(\mbox{R}_x)$ is given by
\begin{eqnarray*}
\mathbb{P}(\mbox{R}_x) &=& \{ [u]=[u_1 +u_2] \in \mathbb{P}(\Lambda^2 T_x M \otimes \C):\, \Lambda^2 g_x\Big(A(u_1),u_1\Big)+\\
&&+2\Lambda^2g_x\Big(B (u_2),u_1\Big)+\Lambda^2g_x\Big(C(u_2),u_2\Big) =0\}.
\end{eqnarray*}

We would like to describe the intersection of $\mathbb{P}(\mbox{R}_x)$ with $\mathbb{P}\mbox{-span}(\pl(t_+),\pl(t_-))$. As explained previously, a point on the line $\mathbb{P}\mbox{-span}(\pl(t_+),\pl(t_-))$ is expressed as $[\lambda T_+ + \mu T_-]$. Let us set $u_1=\lambda T_+\in \Lambda^{2}_{+}(T_x M \otimes \C)$ and $u_2=\mu T_- \in \Lambda^{2}_{-}(T_x M \otimes \C)$. We obtain, that

\begin{comment}
Furthermore the point lies in the intersection of the quadric and the line if and only if 
$$\mbox{R}_x(\lambda T_+ + \mu T_-)=0.$$
Furthermore by the fact that $\pi \Big( (\mbox{id}_{\mathbb{P}(g_x)} \times \mbox{pl})  \big( \mathbb{P}(\mathcal{T}) \big) \Big)= \mbox{pl}(\mbox{Gr}_1(\mathbb{P}(T_x M \otimes \C))) \cap \mathbb{P}(\Lambda^2 g_x)$, we have that
$$\Lambda^2 g_x (\lambda T_+,\lambda T_+ )=\Lambda^2 g_x(\mbox{Id}_{\Lambda_+}(\lambda T_+),\lambda T_+)=0$$
and
$$\Lambda^2 g_x (\mu T_-,\mu T_- )=\Lambda^2 g_x(\mbox{Id}_{\Lambda_-}(\mu T_-),\mu T_-)=0.$$
\end{comment}

\begin{eqnarray*}
&&\Lambda^2 g_x\Big(A(\lambda T_+),\lambda T_+ \Big)+2\Lambda^2g_x\Big(B (\mu T_-),\lambda T_+ \Big)+\Lambda^2g_x\Big(C(\mu T_-),\mu T_- \Big)=0\\
& \Rightarrow & \lambda^2 \Lambda^2 g_x\Big( A (T_+), T_+\Big)+2\lambda \mu \Lambda^2 g_x\Big(B (T_-),T_+ \Big)+\mu^2 \Lambda^2g_x\Big(C(T_-),T_-\Big)=0 \\
& \Rightarrow &  \lambda^2 \Lambda^2 g_x\Big( \big( \mathfrak{W}_+ + \frac{\mbox{scal}}{12}\mbox{Id}_{\Lambda_+}  \big) (T_+), T_+\Big)+2\lambda \mu \Lambda^2 g_x\Big(B (T_-),T_+ \Big)+\\
&&+\mu^2 \Lambda^2g_x\Big( \big( \mathfrak{W}_- + \frac{\mbox{scal}}{12}\mbox{Id}_{\Lambda_-}  \big) (T_-),T_-\Big)=0\\
& \Rightarrow & \lambda^2 \Lambda^2 g_x\Big( \mathfrak{W}_+ (T_+), T_+\Big)+ \lambda^2 \frac{\mbox{scal}}{12}\Lambda^2 g_x\Big( T_+, T_+\Big)+2\lambda \mu \Lambda^2 g_x\Big(B (T_-),T_+ \Big)+\\
&&+\mu^2 \Lambda^2g_x\Big(\mathfrak{W}_- (T_-),T_-\Big)+\mu^2 \frac{\mbox{scal}}{12}\Lambda^2g_x\Big(T_-,T_-\Big)=0 \\
& \Rightarrow & \lambda^2 \Lambda^2 g_x\Big( \mathfrak{W}_+ (T_+), T_+\Big)+2\lambda \mu \Lambda^2 g_x\Big(B (T_-),T_+ \Big)+\mu^2 \Lambda^2g_x\Big(\mathfrak{W}_- (T_-),T_-\Big)=0,
\end{eqnarray*}
where $\mathfrak{W}_+$ and $\mathfrak{W}_-$ correspond to the operators associated to $W_+$ and $W_-$ respectively. Notice, that in the last implication we are using the fact, that $\pi \Big( (\mbox{id}_{\mathbb{P}(g_x)} \times \mbox{pl})  \big( \mathbb{P}(\mathcal{T}) \big) \Big)= \mbox{pl}(\mbox{Gr}_1(\mathbb{P}(T_x M \otimes \C))) \cap \mathbb{P}(\Lambda^2 g_x)$.

\begin{comment}
Notice, that we write $A-\frac{\mbox{scal}}{12}\mbox{Id}_{\Lambda_+}$ and $C - \frac{\mbox{scal}}{12}\mbox{Id}_{\Lambda_-}$, because of the following  observation:  $\mbox{Id}_{\Lambda_{\pm}}$ corresponds to the operator associated to the (2,0)-tensor $ \Lambda^2(g_x)=\frac{1}{2} g \varowedge g $ (we consider in this context $g \varowedge g$ as a (2,0)-tensor.

We obtain, that
\begin{equation*}
\lambda^2 \Lambda^2 g_x\Big( A (T_+), T_+\Big)+2\lambda \mu \Lambda^2 g_x\Big(B (T_-),T_+ \Big)+\mu^2 \Lambda^2g_x\Big(C(T_-),T_-\Big)=0.
\end{equation*}
\end{comment}

By assuming that $\mu \neq 0$ and setting $s=\frac{\lambda}{\mu}$ we obtain a quadratic equation in the variable $s$ given by
\begin{equation}\label{eq: intersection of P(R_x) and line equation}
\Lambda^2 g_x\Big(\mathfrak{W}_+(T_+), T_+\Big) s^2+2\Lambda^2 g_x\Big(B (T_-),T_+ \Big)s+\Lambda^2g_x\Big(\mathfrak{W}_-(T_-),T_-\Big)=0.
\end{equation}
We can consider the previous equation naturally, as an equation that determines $S_x$.
The discriminant of the equation is given by
$$\Delta=4 \Big( \Lambda^2 g_x(B (T_-),T_+) \Big)^2-4\Lambda^2 g_x\Big(\mathfrak{W}_+(T_+), T_+) \Lambda^2g_x\Big(\mathfrak{W_-}(T_-),T_-\Big).$$
Thus there are three possible cases for the intersection of the quadric and the line.
\begin{enumerate}[(i)]
\item If $\Delta \neq 0$, then the intersection consists of exactly two distinct points:
\begin{itemize}
\item $\mathbb{P}\mbox{-span}(\pl(t_+),\pl(t_-)) \cap \mathbb{P}(\mbox{R}_x)=\{ \pl(l),\pl(l^{\prime})\},$
where $\pl(l),\pl(l^{\prime}) \neq \pl(t_+), \pl(t_-)$. Then
$$\tilde{\pi}^{-1}(\pl(l))=(t_{+} \cap t_{-},\pl(l)), \, \tilde{\pi}^{-1}(\pl(l^{\prime}))=(t_+ \cap t_-,\pl(l^{\prime}))$$
and
$$\tilde{\tau}^{-1}(t_+ \cap t_-)= \{ \tilde{\pi}^{-1}(\pl(l)), \tilde{\pi}^{-1}(\pl(l^{\prime}))\}$$
are two distinct points. Both these points are nonsingular points of $\tilde{S}_x$.
\item $\mathbb{P}\mbox{-span}(\pl(t_{+}^{i}),\pl(t_-))  \cap \mathbb{P}(\mbox{R}_x)=\{ \pl(t_{+}^{i}), \pl(l)\},$
for some $i=1,...,4$, where $\pl(l)\neq \pl(t_{+}^{i}), \pl(t_-)$. Then
$$\tilde{\tau}^{-1}(t_{+}^{i} \cap t_-)= \{ (t_+ \cap t_-,\pl(t_{+}^{i})),  \tilde{\pi}^{-1}(\pl(l))\},$$
are two distinct points. Both these points are nonsingular points of $\tilde{S}_x$.
\item  $\mathbb{P}\mbox{-span}(\pl(t_+),\pl(t_{-}^{j})) \cap \mathbb{P}(\mbox{R}_x)=\{\pl(t_{-}^{j}), \pl(l)\},$
for some $j=1,...,4$, where $\pl(l) \neq \pl(t_+), \pl(t_{-}^{j})$. Then
$$\tilde{\tau}^{-1}(t_+ \cap t_{-}^{j})= \{ (t_+ \cap t_-, \pl(t_{-}^{j})),  \tilde{\pi}^{-1}(\pl(l))\},$$
are two distinct points. Both these points are nonsingular points of $\tilde{S}_x$.
\item $\mathbb{P}\mbox{-span}(\pl(t_{+}^{i}),\pl(t_{-}^{j}))  \cap \mathbb{P}(\mbox{R}_x)=\{ \pl(t_{+}^{i}), \pl(t_{-}^{j})\}$, for some $i,j=1,...,4$. Then 
$$\tilde{\tau}^{-1}(t_{+}^{i} \cap t_{-}^{j})= \{ (t_{+}^{i} \cap t_{-}^{j}, \pl(t_{+}^{i})),  (t_{+}^{i} \cap t_{-}^{j},\pl(t_{-}^{j}))\},$$
are two distinct points. Both these points are nonsingular points of $\tilde{S}_x$.
\end{itemize}
\item If $\Delta=0$, but not all coefficients are equal to zero,  then the line has exactly one double point in common with the quadric $\mathbb{P}(\mbox{R}_x)$, which is possible if and only if the line is tangent to the quadric  at that point:
\begin{itemize}
\item $\mathbb{P}\mbox{-span}(\pl(t_+),\pl(t_-))  \cap \mathbb{P}(\mbox{R}_x)=\{\pl(l)\}$, where $\pl(l) \neq \pl(t_+),\pl(t_-)$. In this case $\mathbb{P}\mbox{-span}(\pl(t_+),\pl(t_-))$  is tangent to the quadric $\mathbb{P}(R_x)$ at the point $\pl(l)$. Then 
$$\tilde{\tau}^{-1}(t_+ \cap t_-)= \{\tilde{\pi}^{-1}(\pl(l))\}.$$
Obviously in this case $t_+ \cap t_-$ corresponds to a branching point and $\tilde{\pi}^{-1}(\pl(l))$ is a ramification point.
\begin{comment}
We can write down then the equation for the branching curve $\Gamma_x$ at the point $t_+ \cap t_-$ by using the expressions (\ref{eq: homogeneous coordinates of t_+ in P^5}), (\ref{eq: homogeneous coordinates of t_ - in P^5}) and the fact that for such a choice of basis $\Lambda^2g_x$ is just the identity matrix. The point $\tilde{\pi}^{-1}(\pl(l))$ is a singular point of $\tilde{S}_x$ if and only if not only $\mathbb{P}\mbox{-span}(\pl(t_+,\pl(t_-))) $ is tangent to $\mathbb{P}(\mbox{R}_x)$ , but also the whole $\tilde{\pi}(\mathbb{P}(\mathcal{T}))$ is tangent to $\mathbb{P}(\mbox{R}_x)$ at the point $\pl(l)$, i.e. the tangent space of the one belongs to the tangent space of the other.
\end{comment}
\end{itemize}
\item If $\Delta=0$ and all coefficients are simultaneously equal to zero, then the line lies entirely in $\mathbb{P}(\mbox{R}_x)$:
\begin{itemize}
\item $\mathbb{P}\mbox{-span}(\pl(t_{+}^{i}),\pl(t_{-}^{j})) \subset \mathbb{P}(\mbox{R}_x)$, for some $i,j=1,...,4$. Then
\begin{eqnarray*}
\tilde{\tau}^{-1}(t_{+}^{i} \cap t_{-}^{j})&=& \{ \tilde{\pi}^{-1} (\pl(l)): \; \pl(l) \in \mathbb{P}\mbox{-span}(\pl(t_{+}^{i}),\pl(t_{-}^{j}))  \setminus \{ \pl(t_{+}^{i}),\pl(t_{-}^{j}) \} \} \cup \\
&& \cup \{(t_{+}^{i} \cap t_{-}^{j},\pl(t_{+}^{i})) \} \cup \{(t_{+}^{i} \cap t_{-}^{j},\pl(t_{-}^{j}))\} =:  \mathbb{P}^{1}_{t_{+}^{i} \cap t_{-}^{j}},
\end{eqnarray*}
where $ \mathbb{P}^{1}_{t_{+}^{i} \cap t_{-}^{j}} \cong \mathbb{P}\mbox{-span}(\pl(t_{+}^{i}),\pl(t_{-}^{j})) \cong \mathbb{P}^1$, since $\tilde{\pi}$ maps the curve $\tilde{\tau}^{-1}(t_{+}^{i} \cap t_{-}^{j})$ one to one onto the singular line $\mathbb{P}\mbox{-span}(\pl(t_{+}^{i}),\pl(t_{-}^{j}))$.  Here $t_{+}^{i} \cap t_{-}^{j}$ corresponds again to a branching point and in this special case the branching curve $\Gamma_x \subset \mathbb{P}(g_x)$ at the point $t_{+}^{i} \cap t_{-}^{j}$ is singular.  
\begin{comment}
We will call these points, singular points of the second type. We are highly interested under which assumptions the points of the curve $ \mathbb{P}^{1}_{t_+ \cap t_-}$ correspond to nonsingular points of $\tilde{S}_x$. This is true when 
\begin{enumerate}[(i)]
\item $\mathbb{P}(\mbox{R}_x)$ intersects $\tilde{\pi}(\mathbb{P}(\mathcal{T}))$ transeversaly along the line $\mathbb{P}\mbox{-span}(\pl(t_+),\pl(t_-)) $ at nonsingular points of $\tilde{\pi}(\mathbb{P}(\mathcal{T}))$ and\\
\item $\mathbb{P}(\mbox{R}_x)$ intersects $\mathcal{C}_+$ and $\mathcal{C}_-$ transeversaly at the points $\pl(t_+)$ and $\pl(t_-)$  as well.
\end{comment}
\end{itemize}
\end{enumerate}
Thus the branching curve is described by
\begin{eqnarray*}
\Gamma_x &=& \{ ([a^1,a^2],[b^1,b^2]) \in \mathbb{P}(S^{-}_{x}) \times \mathbb{P}(S^{+}_{x}): \\ \nonumber
&&\Big( \Lambda^2 g_x(B (T_-),T_+) \Big)^2-\Lambda^2 g_x\Big(\mathfrak{W}_+(T_+), T_+\Big) \Lambda^2g_x\Big(\mathfrak{W}_-(T_-),T_-\Big)=0   \}.
\end{eqnarray*}

\begin{rem}
The branching curve will serve as our local invariant in this text. Precisely, we will use this local invariant in oder to obtain a characterization for the singularity models for Type I singularities for four dimensional Ricci flows. The type of the curve is invariant under the choice of basis for $T_x M \otimes \C$. For example, we will see in the next section, that the branching curve associated to a point of $\mathbb{S}^3\times \R$  is a quadruple diagonal and that of $\mathbb{S}^2 \times \mathbb{S}^2$ is a double rectangle.
\end{rem}

The next propositions can be found in Nikulin's paper \cite{Nikulin}.

\begin{prop}\label{th: Prop1}
Assume that  the branching curve $\Gamma_x$ has only finite number of singular points. Then $\tilde{\tau}:\tilde{S}_x \to \mathbb{P}(g_x)$ is a branched double cover for all points $t_+ \cap t_- \in \Gamma_x$, except for the singular points, at which $\tilde{\tau}$ is a branched double cover followed by a blow-up.
\end{prop}
Recall that for a covering map $\tilde{\tau}:\tilde{S}_x \to \mathbb{P}(g_x)$, there exists a homeomorpish $\hat{\sigma}: \tilde{S}_x  \to \tilde{S}_x $, such that $\tilde{\tau} \circ \hat{\sigma} = \tilde{\tau} $, that is to say $\hat{\sigma}$ is a lift of $\tilde{\tau}$. The map $\hat{\sigma}$ is called a deck transformation.

\begin{prop}\label{th: Prop2}
Assume that  the branching curve $\Gamma_x$ has only finite number of singular points. Then the deck transformation $\hat{\sigma}$ of the branched double cover is everywhere defined on $\tilde{S}_x$.
\end{prop}
Then there are given on $\tilde{S}_x$ nonsingular rational curves (exceptional curves) $\bar{E}_i:=\hat{\sigma}(E_i) \cong \mathbb{P}^1$ and $\bar{F}_j:=\hat{\sigma}(F_j) \cong \mathbb{P}^1$, where $1 \leq i,j\leq 4$. The interested reader can look up Chapter 2.3 of  \cite{Tergiakidis} for further details.

\begin{comment}
\begin{prop}\label{th: Prop3}
$\tilde{S}_x$ is nonsingular if and only if the branch curve $\Gamma_x$ has only singular points of the second type and they are ordinary double points.
\end{prop}
\end{comment}

\section{Examples of local invariants}

In this section we compute the branching curve (our local invariant) for solutions $(M,g(t))$. We look at the examples of $\mathbb{S}^3 \times \R$, $\mathbb{S}^2 \times \mathbb{S}^2$, $\mathbb{S}^2 \times \R^2$, $\mathbb{P}^2$ and $ \mathbb{S}^4$. More explicit computation can be found in \cite{Tergiakidis}.

\subsection{The example of $(\mathbb{S}^3 \times \R,g(t))$}

The initial metric $g_0$ (with respect to spherical coordinates on the $\mathbb{S}^3$ factor)  is given by
$$g_0=d \phi_{1}^{2}+\sin^2 \phi_1 d \phi_{2}^{2}+ \sin^2 \phi_1 \sin^2 \phi_2 d \phi_{3}^{2}+dx^2.$$
Recall that the Ricci flow evolves each factor of a product metric seperately and if we use the formula for the evolution of the round metric on the sphere, we obtain that a solution to the Ricci flow is given by
$$g(t)=(1-4t)d \phi_{1}^{2}+(1-4t)\sin^2 \phi_1 d \phi_{2}^{2}+(1-4t) \sin^2 \phi_1 \sin^2 \phi_2 d \phi_{3}^{2}+dx^2.$$
The set $\Big\{ \frac{\partial}{\partial \phi_1},\frac{\partial}{\partial \phi_2} \,\frac{\partial}{\partial \phi_3} ,\frac{\partial}{\partial x}  \Big\}$
constitutes a basis for $T_x M$. We obtain a time-dependent orthonormal frame, with respect to which the metric becomes diagonal by setting
$$\Big\{ e_a=\frac{1}{\sqrt{1-4t}}\frac{\partial}{\partial \phi_1},e_b=\frac{1}{\sqrt{1-4t}\sin\phi_1}\frac{\partial}{\partial \phi_2},e_c= \frac{1}{\sqrt{1-4t} \sin\phi_1 \sin \phi_2}\frac{\partial}{\partial \phi_3} ,e_d=\frac{\partial}{\partial x}   \Big\},$$
where $e_{\alpha}:=e_{\alpha}(x,t)$, $\alpha=a,b,c,d$.

The components of the $(4,0)$-Riemann curvature tensor are given by
$$\tensor{R}{_{abba}}=\tensor{R}{_{acca}}=\tensor{R}{_{bccb}}=\frac{1}{1-4t}.$$

We are now in position to compute the scalar curvature. The Ricci curvature is
$$\tensor{R}{_{aa}}=\tensor{R}{_{bb}}=\tensor{R}{_{cc}}=\frac{2}{1-4t}.$$
Thus $\mbox{scal}=\frac{6}{1-4t}$ and $\frac{\mbox{scal}}{12}=\frac{1}{2(1-4t)}$. Furthermore
\begin{eqnarray*}
\Lambda^2g_x(\cop(f_{1}^{+}),f_{1}^{+}) &=& \Lambda^2g_x(\cop(\frac{1}{\sqrt{2}} (e_a \wedge e_b + e_c \wedge e_d)),\frac{1}{\sqrt{2}} (e_a \wedge e_b + e_c \wedge e_d)) \\
&=&\frac{1}{2}(R_{abba}+R_{abdc}+R_{cdba}+R_{cddc})\\
&=&\frac{1}{2(1-4t)}.
\end{eqnarray*}

\begin{eqnarray*}
\Lambda^2g_x(\cop(f_{2}^{+}),f_{2}^{+}) &=& \Lambda^2g_x(\cop(\frac{1}{\sqrt{2}} (e_a \wedge e_c - e_b \wedge e_d)),\frac{1}{\sqrt{2}} (e_a \wedge e_c - e_b \wedge e_d)) \\
&=&\frac{1}{2}(R_{acca}-R_{acdb}+R_{bdca}+R_{bddb})\\
&=&\frac{1}{2(1-4t)}.
\end{eqnarray*}

\begin{eqnarray*}
\Lambda^2g_x(\cop(f_{3}^{+}),f_{3}^{+}) &=& \Lambda^2g_x(\cop(\frac{1}{\sqrt{2}} (e_a \wedge e_d + e_b \wedge e_c)),\frac{1}{\sqrt{2}} (e_a \wedge e_d + e_b \wedge e_c)) \\
&=&\frac{1}{2}(R_{adda}-R_{adcb}+R_{bcda}+R_{bccb})\\
&=&\frac{1}{2(1-4t)}.
\end{eqnarray*}

\begin{eqnarray*}
\Lambda^2g_x(\cop(f_{1}^{-}),f_{1}^{+}) &=& \Lambda^2g_x(\cop(\frac{1}{\sqrt{2}} (e_a \wedge e_b - e_c \wedge e_d)),\frac{1}{\sqrt{2}} (e_a \wedge e_b + e_c \wedge e_d)) \\
&=&\frac{1}{2}(R_{abba}+R_{abdc}-R_{cdba}-R_{cddc})\\
&=&\frac{1}{2(1-4t)}.
\end{eqnarray*}

\begin{eqnarray*}
\Lambda^2g_x(\cop(f_{2}^{-}),f_{2}^{+}) &=& \Lambda^2g_x(\cop(\frac{1}{\sqrt{2}} (e_a \wedge e_c + e_b \wedge e_d)),\frac{1}{\sqrt{2}} (e_a \wedge e_c - e_b \wedge e_d)) \\
&=&\frac{1}{2}(R_{acca}-R_{acdb}+R_{bdca}-R_{bddb})\\
&=&\frac{1}{2(1-4t)}.
\end{eqnarray*}

\begin{eqnarray*}
\Lambda^2g_x(\cop(f_{3}^{-}),f_{3}^{+}) &=& \Lambda^2g_x(\cop(\frac{1}{\sqrt{2}} (e_a \wedge e_d - e_b \wedge e_c)),\frac{1}{\sqrt{2}} (e_a \wedge e_d + e_b \wedge e_c)) \\
&=&\frac{1}{2}(R_{adda}+R_{adcb}-R_{bcda}-R_{bccb})\\
&=&-\frac{1}{2(1-4t)}.
\end{eqnarray*}

\begin{eqnarray*}
\Lambda^2g_x(\cop(f_{1}^{-}),f_{1}^{-}) &=& \Lambda^2g_x(\cop(\frac{1}{\sqrt{2}} (e_a \wedge e_b - e_c \wedge e_d)),\frac{1}{\sqrt{2}} (e_a \wedge e_b - e_c \wedge e_d)) \\
&=&\frac{1}{2}(R_{abba}-R_{abdc}-R_{cdba}+R_{cddc})\\
&=&\frac{1}{2(1-4t)}.
\end{eqnarray*}

\begin{eqnarray*}
\Lambda^2g_x(\cop(f_{2}^{-}),f_{2}^{-}) &=& \Lambda^2g_x(\cop(\frac{1}{\sqrt{2}} (e_a \wedge e_c + e_b \wedge e_d)),\frac{1}{\sqrt{2}} (e_a \wedge e_c + e_b \wedge e_d)) \\
&=&\frac{1}{2}(R_{acca}+R_{acdb}+R_{bdca}+R_{bddb})\\
&=&\frac{1}{2(1-4t)}.
\end{eqnarray*}

\begin{eqnarray*}
\Lambda^2g_x(\cop(f_{3}^{-}),f_{3}^{-}) &=& \Lambda^2g_x(\cop(\frac{1}{\sqrt{2}} (e_a \wedge e_d - e_b \wedge e_c)),\frac{1}{\sqrt{2}} (e_a \wedge e_d - e_b \wedge e_c)) \\
&=&\frac{1}{2}(R_{adda}-R_{adcb}-R_{bcda}+R_{bccb})\\
&=&\frac{1}{2(1-4t)}.
\end{eqnarray*}
This means that the matrices of the bilinear forms 
$$\Lambda^2g_x \Big( \mathfrak{W}_+(\cdot),\cdot\Big)=\Lambda^2g_x \Big( \big( A -  \frac{\mbox{scal}}{12}\mbox{Id}_{\Lambda_+}  \big)(\cdot),\cdot\Big)$$ 
and 
$$\Lambda^2g_x \Big(\mathfrak{W}_-(\cdot),\cdot\Big)=\Lambda^2g_x \Big( \big( C -  \frac{\mbox{scal}}{12}\mbox{Id}_{\Lambda_-}  \big)(\cdot),\cdot\Big)$$ 
are given by

\begin{eqnarray*}
\begin{bmatrix}
0 && 0 && 0\\
0 && 0 && 0\\
0 && 0 && 0
\end{bmatrix}
\end{eqnarray*}
and  the matrix of the bilinear form $\Lambda^2g_x \Big( B(\cdot),\cdot\Big)$ by

\begin{eqnarray*}
\begin{bmatrix}
\frac{1}{2(1-4t)} && 0 && 0\\
0 && \frac{1}{2(1-4t)} && 0\\
0 && 0 &&- \frac{1}{2(1-4t)}
\end{bmatrix}
.
\end{eqnarray*}
We are now going to compute the branching curve. Obviously
$$\Lambda^2g_x \Big( \mathfrak{W}_+(T_+),T_+ \Big)=\Lambda^2g_x \Big( \mathfrak{W}_-(T_-),T_- \Big)=0$$
and
$$\Lambda^2g_x \Big( B(T_-),T_+ \Big) = -\frac{2}{1-4t} a^1 a^2 b^1 b^2+\frac{1}{1-4t}(a^{1})^2(b^{2})^2+\frac{1}{1-4t}(a^{2})^2(b^{1})^2,$$
where  
$$T_-= 2 ib^1b^2 f_{1}^{-} +i\{(b^{2})^{2}-(b^{1})^{2}\}f_{2}^{-} +[(b^{1})^{2}+(b^{2})^{2}]f_{3}^{-}$$
and
$$T_+=2 ia^1a^2f_{1}^{+}+ i\{(a^{2})^{2}-(a^{1})^{2}\}f_{2}^{+} +[ -(a^{1})^{2}-(a^{2})^{2})]f_{3}^{+}.$$
We compute that 
$$\Big[ \Lambda^2g_x \Big( B(T_-),T_+\Big) \Big]^2=\frac{1}{(1-4t)^2}(a^1b^2-a^2b^1)^4.$$
Thus
$$\Gamma_x=\{ ([a^1,a^2],[b^1,b^2]) \in \mathbb{P}(S^{-}_{x}) \times \mathbb{P}(S_{x}^{+}): \, (a^1b^2-a^2b^1)^4=0  \}.$$
This curve is never smooth and has multiplicity four. Notice, that in this the branching curve represents geometrically a quadruple diagonal.

\subsection{The example of $(\mathbb{S}^2 \times \mathbb{S}^2,g(t))$}

The initial metric $g_0$ (with respect to spherical coordinates on both $\mathbb{S}^2$ factors)  is given by
$$g_0=d \phi_{1}^{2}+\sin^2 \phi_1 d \phi_{2}^{2}+ d \psi_{1}^{2}+\sin^2 \psi_1 d \psi_{2}^2.$$
Now a solution to the Ricci flow is given by
$$g(t)=(1-2t)d \phi_{1}^{2}+(1-2t)\sin^2 \phi_1 d \phi_{2}^{2}+ (1-2t)d \psi_{1}^{2}+(1-2t)\sin^2 \psi_1 d \psi_{2}^2.$$
The set $\Big\{ \frac{\partial}{\partial \phi_1},\frac{\partial}{\partial \phi_2} \,\frac{\partial}{\partial \psi_1} ,\frac{\partial}{\partial \psi_2}  \Big\}$
constitutes a basis for $T_x M$. We obtain an orthonormal frame, with respect to which the metric becomes diagonal by setting
$$\Big\{ e_a=\frac{1}{\sqrt{1-2t}}\frac{\partial}{\partial \phi_1},e_b=\frac{1}{\sqrt{1-2t}\sin\phi_1}\frac{\partial}{\partial \phi_2},e_c= \frac{1}{\sqrt{1-2t}}\frac{\partial}{\partial \psi_1} ,e_d=\frac{1}{\sqrt{1-2t}\sin\psi_1}\frac{\partial}{\partial \psi_2}  \Big\}.$$

The components of the $(4,0)$-Riemann curvature tensor are given by
$$\tensor{R}{_{abba}}=\tensor{R}{_{cddc}}=\frac{1}{1-2t}.$$

The Ricci curvature is
$$\tensor{R}{_{aa}}=tensor{R}{_{bb}}=\tensor{R}{_{cc}}=\tensor{R}{_{dd}}=\frac{1}{1-2t}.$$
Thus $\mbox{scal}=\frac{4}{1-2t}$ and $\frac{\mbox{scal}}{12}=\frac{1}{3(1-2t)}$. Furthermore
\begin{eqnarray*}
\Lambda^2g_x(\cop(f_{1}^{+}),f_{1}^{+}) &=& \Lambda^2g_x(\cop(\frac{1}{\sqrt{2}} (e_a \wedge e_b + e_c \wedge e_d)),\frac{1}{\sqrt{2}} (e_a \wedge e_b + e_c \wedge e_d)) \\
&=&\frac{1}{2}(R_{abba}+R_{abdc}+R_{cdba}+R_{cddc})\\
&=&\frac{1}{2(1-2t)}.
\end{eqnarray*}

\begin{eqnarray*}
\Lambda^2g_x(\cop(f_{1}^{-}),f_{1}^{-}) &=& \Lambda^2g_x(\cop(\frac{1}{\sqrt{2}} (e_a \wedge e_b - e_c \wedge e_d)),\frac{1}{\sqrt{2}} (e_a \wedge e_b - e_c \wedge e_d)) \\
&=&\frac{1}{2}(R_{abba}-R_{abdc}-R_{cdba}+R_{cddc})\\
&=&\frac{1}{2(1-2t)}.
\end{eqnarray*}
This means that the matrices of the bilinear forms 
$$\Lambda^2g_x \Big( \mathfrak{W}_+(\cdot),\cdot\Big)=\Lambda^2g_x \Big( \big( A -  \frac{\mbox{scal}}{12}\mbox{Id}_{\Lambda_+}  \big)(\cdot),\cdot\Big)$$ 
and 
$$\Lambda^2g_x \Big(\mathfrak{W}_-(\cdot),\cdot\Big)=\Lambda^2g_x \Big( \big( C -  \frac{\mbox{scal}}{12}\mbox{Id}_{\Lambda_-}  \big)(\cdot),\cdot\Big)$$ 
are given by

\begin{eqnarray*}
\begin{bmatrix}
\frac{1}{6(1-2t)}&& 0 && 0\\
0 && -\frac{1}{3(1-2t)} && 0\\
0 && 0 && -\frac{1}{3(1-2t)}
\end{bmatrix}
\end{eqnarray*}
and the matrix of the bilinear form $\Lambda^2g_x \Big( B(\cdot),\cdot\Big)$ by
\begin{eqnarray*}
\begin{bmatrix}
0 && 0 && 0\\
0 && 0 && 0\\
0 && 0 &&0
\end{bmatrix}
.
\end{eqnarray*}
We are now going to compute the branching curve. 

$$\Lambda^2g_x \Big( \mathfrak{W}_+(T_+),T_+ \Big)=-\frac{2}{3(1-2t)}(a^1)^2(a^2)^2,$$
$$\Lambda^2g_x \Big( \mathfrak{W}_-(T_-),T_- \Big)=-\frac{2}{3(1-2t)}(b^1)^2(b^2)^2.$$
where  
$$T_-= 2 ib^1b^2 f_{1}^{-} +i\{(b^{2})^{2}-(b^{1})^{2}\}f_{2}^{-} +[(b^{1})^{2}+(b^{2})^{2}]f_{3}^{-}$$
and
$$T_+=2 ia^1a^2f_{1}^{+}+ i\{(a^{2})^{2}-(a^{1})^{2}\}f_{2}^{+} +[ -(a^{1})^{2}-(a^{2})^{2})]f_{3}^{+}.$$
Thus

$$\Gamma_x=\{ ([a^1,a^2],[b^1,b^2]) \in \mathbb{P}(S^{-}_{x}) \times \mathbb{P}(S_{x}^{+}) : \, (a^1a^2b^1b^2)^2=0  \}.$$
Notice, that in this the branching curve represents geometrically a double rectangle.

\subsection{The example of $(\mathbb{S}^2 \times \R^2,g(t))$}

The initial metric $g_0$ (with respect to spherical coordinates on the $\mathbb{S}^2$ factor)  is given by
$$g_0=d \phi_{1}^{2}+\sin^2 \phi_1 d \phi_{2}^{2}+ dx^2+dy^2.$$
In this case a solution to the Ricci flow is given by
$$g(t)=(1-2t)d \phi_{1}^{2}+(1-2t)\sin^2 \phi_1 d \phi_{2}^{2}+dx^2+dy^2.$$
The set $\Big\{ \frac{\partial}{\partial \phi_1},\frac{\partial}{\partial \phi_2} \,\frac{\partial}{\partial x}  ,\frac{\partial}{\partial y}  \Big\}$
constitutes a basis for $T_x M$. We obtain an orthonormal frame, with respect to which the metric becomes diagonal by setting
$$\Big\{ e_a=\frac{1}{\sqrt{1-2t}}\frac{\partial}{\partial \phi_1},e_b=\frac{1}{\sqrt{1-2t}\sin\phi_1}\frac{\partial}{\partial \phi_2},e_c=\frac{\partial}{\partial x} ,e_d=\frac{\partial}{\partial y}   \Big\}$$

The components of the $(4,0)$-Riemann curvature tensor are given by
$$\tensor{R}{_{abba}}=\frac{1}{1-2t}.$$
The Ricci curvature is
$$\tensor{R}{_{aa}}=\tensor{R}{_{bb}}=\frac{1}{1-2t}.$$
Thus the scalar curvature is given by $\mbox{scal}=\frac{2}{1-2t}$ and $\frac{\mbox{scal}}{12}=\frac{1}{6(1-2t)}$.
Furthermore
\begin{eqnarray*}
\Lambda^2g_x(\cop(f_{1}^{+}),f_{1}^{+}) &=& \Lambda^2g_x(\cop(\frac{1}{\sqrt{2}} (e_a \wedge e_b + e_c \wedge e_d)),\frac{1}{\sqrt{2}} (e_a \wedge e_b + e_c \wedge e_d)) \\
&=&\frac{1}{2}(R_{abba}+R_{abdc}+R_{cdba}+R_{cddc})\\
&=&\frac{1}{2(1-2t)}.
\end{eqnarray*}

\begin{eqnarray*}
\Lambda^2g_x(\cop(f_{1}^{-}),f_{1}^{+}) &=& \Lambda^2g_x(\cop(\frac{1}{\sqrt{2}} (e_a \wedge e_b - e_c \wedge e_d)),\frac{1}{\sqrt{2}} (e_a \wedge e_b + e_c \wedge e_d)) \\
&=&\frac{1}{2}(R_{abba}+R_{abdc}-R_{cdba}-R_{cddc})\\
&=&\frac{1}{2(1-2t)}.
\end{eqnarray*}

\begin{eqnarray*}
\Lambda^2g_x(\cop(f_{1}^{-}),f_{1}^{-}) &=& \Lambda^2g_x(\cop(\frac{1}{\sqrt{2}} (e_a \wedge e_b - e_c \wedge e_d)),\frac{1}{\sqrt{2}} (e_a \wedge e_b - e_c \wedge e_d)) \\
&=&\frac{1}{2}(R_{abba}-R_{abdc}-R_{cdba}+R_{cddc})\\
&=&\frac{1}{2(1-2t)}.
\end{eqnarray*}
This means that the matrices of the bilinear forms 
$$\Lambda^2g_x \Big( \mathfrak{W}_+(\cdot),\cdot\Big)=\Lambda^2g_x \Big( \big( A -  \frac{\mbox{scal}}{12}\mbox{Id}_{\Lambda_+}  \big)(\cdot),\cdot\Big)$$ 
and 
$$\Lambda^2g_x \Big(\mathfrak{W}_-(\cdot),\cdot\Big)=\Lambda^2g_x \Big( \big( C -  \frac{\mbox{scal}}{12}\mbox{Id}_{\Lambda_-}  \big)(\cdot),\cdot\Big)$$ 
are given by

\begin{eqnarray*}
\begin{bmatrix}
\frac{1}{3(1-2t)} && 0 && 0\\
0 && -\frac{1}{6(1-2t)}  && 0\\
0 && 0 && -\frac{1}{6(1-2t)} 
\end{bmatrix}
\end{eqnarray*}
and the matrix of the bilinear form $\Lambda^2g_x \Big( B(\cdot),\cdot\Big)$ by

\begin{eqnarray*}
\begin{bmatrix}
\frac{1}{2(1-2t)} && 0 && 0\\
0 && 0 && 0\\
0 && 0 &&0
\end{bmatrix}
.
\end{eqnarray*}

We are now going to compute the branching curve. 

$$\Lambda^2g_x \Big( \mathfrak{W}_+(T_+),T_+\Big)=-\frac{2}{(1-2t)}(a^1)^2(a^2)^2,$$
$$\Lambda^2g_x \Big( \mathfrak{W}_-(T_-),T_- \Big)=-\frac{2}{(1-2t)}(b^1)^2(b^2)^2,$$
$$\Lambda^2g_x \Big( B(T_-),T_+ \Big) = -\frac{2}{1-2t} a^1 a^2 b^1 b^2.$$
where  
$$T_-= 2 ib^1b^2 f_{1}^{-} +i\{(b^{2})^{2}-(b^{1})^{2}\}f_{2}^{-} +[(b^{1})^{2}+(b^{2})^{2}]f_{3}^{-}$$
and
$$T_+=2 ia^1a^2f_{1}^{+}+ i\{(a^{2})^{2}-(a^{1})^{2}\}f_{2}^{+} +[ -(a^{1})^{2}-(a^{2})^{2})]f_{3}^{+}.$$
We compute that 
$$\Big[ \Lambda^2g_x \Big( B(T_-),T_+\Big) \Big]^2=\frac{4}{(1-2t)^2}(a^1a^2b^1b^2)^2=\Lambda^2g_x \Big( \mathfrak{W}_+(T_+),T_+\Big)\Lambda^2g_x \Big( \mathfrak{W}_-(T_-),T_- \Big).$$
One observes, that in this case the branching curve doesn't exist.

\subsection{The example of $(\mathbb{P}^2,g(t))$}

The initial metric $g_{FS}$  is the Fubini-Study metric.
A solution to the Ricci flow is given by
$$g(t)=(1-2 \kappa t) g_{FS},$$
where $\kappa>0$.
By working exactly in the same way as is the previous examples one can obtain that the matrix of the bilinear form 
$$\Lambda^2g_x \Big( \mathfrak{W}_+(\cdot),\cdot\Big)=\Lambda^2g_x \Big( (A-\frac{\mbox{scal}}{12} \mbox{Id}_{\Lambda_+})(\cdot),\cdot\Big)$$
 is given by

\begin{eqnarray*}
\begin{bmatrix}
\frac{1}{2(1-2\kappa t)}-\frac{\mbox{scal}}{12} && 0 && 0\\
0 && \frac{1}{2(1-2\kappa t)}-\frac{\mbox{scal}}{12}  && 0\\
0 && 0 && \frac{1}{2(1-2\kappa t)}-\frac{\mbox{scal}}{12} 
\end{bmatrix}
,
\end{eqnarray*}
that of 
$$\Lambda^2g_x \Big( \mathfrak{W}_-(\cdot),\cdot\Big)=\Lambda^2g_x \Big( (C-\frac{\mbox{scal}}{12} \mbox{Id}_{\Lambda_-})(\cdot),\cdot\Big)$$
by
\begin{eqnarray*}
\begin{bmatrix}
\frac{3}{2(1-2\kappa t)} -\frac{\mbox{scal}}{12}&& 0 && 0\\
0 && -\frac{\mbox{scal}}{12} && 0\\
0 && 0 && -\frac{\mbox{scal}}{12}
\end{bmatrix}
\end{eqnarray*}
and finally the matrix of the bilinear form $\Lambda^2g_x \Big( B(\cdot),\cdot\Big)$ by
\begin{eqnarray*}
\begin{bmatrix}
0 && 0 && 0\\
0 && 0 && 0\\
0 && 0 &&0
\end{bmatrix}
.
\end{eqnarray*}
We are now going to compute the branching curve. 

$$\Lambda^2g_x \Big( \mathfrak{W}_+ (T_+),T_+ \Big)=0,$$
$$\Lambda^2g_x \Big( \mathfrak{W}_- (T_-),T_-\Big)=-\Big( \frac{6}{1-2\kappa t} - \frac{\mbox{scal}}{3} \Big)(b^1)^2(b^2)^2 .$$
where  
$$T_-= 2 ib^1b^2 f_{1}^{-} +i\{(b^{2})^{2}-(b^{1})^{2}\}f_{2}^{-} +[(b^{1})^{2}+(b^{2})^{2}]f_{3}^{-}$$
and
$$T_+=2 ia^1a^2f_{1}^{+}+ i\{(a^{2})^{2}-(a^{1})^{2}\}f_{2}^{+} +[ -(a^{1})^{2}-(a^{2})^{2})]f_{3}^{+}.$$
Thus there is no curve and the branching locus is the whole quadric $\mathbb{P}(g_x)$.

\subsection{The example of $(\mathbb{S}^4,g(t))$}

In this section we  show that  the local invarants for the solution $(\mathbb{S}^4 ,g(t))$ do not exist.
The initial metric $g_0$ (with respect to spherical coordinates on $\mathbb{S}^4$)  is given by
$$g_0=d \phi_{1}^{2}+\sin^2 \phi_1 d \phi_{2}^{2}+ sin^2 \phi_1 \sin^2 \phi_2 d \phi_{3}^2+\sin^2 \sin^2 \phi_2 \sin^2 \phi_3 d \phi_{4}^2.$$
A solution to the Ricci flow is given by
$$g(t)=(1-6t)d \phi_{1}^{2}+(1-6t)\sin^2 \phi_1 d \phi_{2}^{2}+(1-6t) sin^2 \phi_1 \sin^2 \phi_2 d \phi_{3}^2+(1-6t)\sin^2 \sin^2 \phi_2 \sin^2 \phi_3 d \phi_{4}^2.$$
Working exactly as in the previous examples ones can compute that matrices of the bilinear forms $\Lambda^2g_x \Big( \mathfrak{W}_+(\cdot),\cdot\Big)$and $\Lambda^2g_x \Big( \mathfrak{W}_-(\cdot),\cdot\Big)$ are given by

\begin{eqnarray*}
\begin{bmatrix}
\frac{1}{3(1-6t)} && 0 && 0\\
0 && \frac{1}{3(1-6t)}  && 0\\
0 && 0 && \frac{1}{3(1-6t)} 
\end{bmatrix}
\end{eqnarray*}
and the matrix of the bilinear form $\Lambda^2g_x \Big( B(\cdot),\cdot\Big)$ by

\begin{eqnarray*}
\begin{bmatrix}
0 && 0 && 0\\
0 && 0 && 0\\
0 && 0 &&0
\end{bmatrix}
,
\end{eqnarray*}
which implies that 

$$\Lambda^2g_x \Big( \mathfrak{W}_+(T_+),T_+\Big)=0,$$
$$\Lambda^2g_x \Big( \mathfrak{W}_-(T_-),T_- \Big)=0,$$
with $T_-$ and $T_+$ as in the previous examples.
Thus in this case the branching doesn't exist.

\section{Type I singularities and the branching curve}

By the results of Naber \cite{Naber} and Enders, M\"uller, Topping \cite{EMT} on Type I singularities for the Ricci flow, it follows that along any sequence of times converging to the finite extinction time $T$, parabolic rescalings will subconverge to a normalized nonflat  gradient shrinking Ricci soliton. In this section we use the construction of Section 2 and apply it to this result, in order to obtain a characterization of the nonflat gradient shrinking solitons in the language of our local invariant. We will need the following lemmas in order to prove our result for Type I singularities.

\begin{lem}\label{lem: scale invariance}
Let $(M^4,g)$ be a $4$-dimensional Riemannian manifold and $x\in M$, such that the branching curve $\Gamma_x$ exists. Then $\Gamma_x$ remains invariant under scalings of the metric by a constant factor.
\end{lem}

\begin{proof}
Let $\kappa$ be some constant factor and and let $\tilde{g}=\kappa g$. Then we know that $\Lambda^2 \tilde{g}=\kappa^2 \Lambda^2 g$, and $\widetilde{\cop}=\frac{1}{\kappa} \cop$.
Then the branching curve is given by
\begin{eqnarray*}
\Gamma_{x}^{\tilde{g}} &=& \{ ([a^1,a^2],[b^1,b^2]) \in \mathbb{P}(S^{-}_{x}) \times \mathbb{P}(S^{+}_{x}) : \, \Big(\kappa^2  \Lambda^2 g_x \Big(\frac{1}{\kappa}B (T_-),T_+ \Big) \Big)^2-\\
&&\kappa^2\Lambda^2 g_x\Big(\frac{1}{\kappa}\mathfrak{W}_+(T_+), T_+\Big)\kappa^2 \Lambda^2g_x\Big( \frac{1}{\kappa}\mathfrak{W}_-(T_-),T_-\Big)=0   \} \\
&=& \{ ([x^0,x^1],[y^0,y^1]) \in  \mathbb{P}(S^{-}_{x}) \times \mathbb{P}(S^{+}_{x}) : \,\kappa^2 \Big(  \Lambda^2 g_x(B (T_-),T_+) \Big)^2-\\
&&\kappa^2\Lambda^2 g_x\Big(\mathfrak{W}_+(T_+), T_+\Big) \Lambda^2g_x\Big(\mathfrak{W}_-(T_-),T_-\Big)=0   \} \\
&=& \Gamma_{x}^{g}.
\end{eqnarray*}
\end{proof}

\begin{lem} \label{lem: convergence curvature}
Let $\{ (M^n, g_i(t),x, F_i(t)) \}_{i \in \mathbb{N}}$, $t \in (\alpha, \omega) \ni 0$ be a sequence of smooth, complete, marked solutions to the Ricci flow, where the time-dependent frame $F_i(t)$ evolves to stay orthonormal. If the sequence  converges to a complete marked solution to the Ricci flow $(M_{\infty}, g_{\infty}(t),x_{\infty}, F_{\infty}(t))$, $t \in (\alpha, \omega) $ as $i \to \infty$, where $F_{\infty}(t)$ evolves to stay orthonormal, then the sequence $\{ (M, \emph{Rm}_{g_i(t)},x, F_i(t)) \}_{i \in \mathbb{N}}$, $t \in (\alpha, \omega) \ni 0$ converges to   $(M_{\infty}, \emph{Rm}_{g_\infty(t)},x_{\infty}, F_{\infty}(t)) $, $t \in (\alpha, \omega) $ as $i \to \infty$.
\end{lem}

\begin{proof}
Let $\{ U_i \}_{i \in \mathbb{N}}$ be an exhaustion of $M_{\infty}$ by open sets with $x_{\infty} \in U_i$ for all $i \in \mathbb{N}$. Furthermore let $\phi_i: U_i \to \phi_i(U_i) \subset M$ be a sequence of diffeomorphisms with $\phi_i(x_{\infty})=x$ and $(\phi_i)_* F_{\infty}(t)=F_i(t)$ for all $i \in \mathbb{N}$ and $t \in (\alpha, \omega)$. We know that $(U_i, \phi_{i}^{*}\Big[ g_i(t)|_{\phi(U_i)} \Big])$ converges in $C^{\infty}$ to $(M_{\infty},g_{\infty}(t))$ uniformly on compact sets in $M_{\infty}$. But uniform convergence of $\phi_{i}^{*}\Big[ g_i(t)|_{\phi(U_i)} \Big]$ to $g_{\infty}(t)$ in $C^k$ for any $k \geq 2$ implies immediately uniform convergence of $\phi_{i}^{*}\Big[ \mbox{Rm}_{g_i(t)}|_{\phi(U_i)} \Big]$ to $\mbox{Rm}_{g_\infty(t)}$ in $C^{k-2}$. This comes from the fact, that the components of the Riemann curvature tensor are determined by the second order (spatial) derivatives of the components of the Riemannian metric tensor. Thus one can deduce that $(U_i, \phi_{i}^{*}\Big[ \mbox{Rm}_{g_i(t)}|_{\phi(U_i)} \Big])$ converges in $C^{\infty}$ to $(M_{\infty},\mbox{Rm}_{g_{\infty}(t)})$ uniformly on compact sets in $M_{\infty}$.
\end{proof}

\begin{rem}
There is a reason behind the fact that we choose to work with an evolving orthonormal frame, which evolves to stay orthonormal. Our goal is to prove Theorem \ref{th: converge of curves}, which states, that convergence of metrics implies convergence of curves. This extra assumption guarantees us the desired extra control over the convergence of branching curves. 
\end{rem}

We are now in position to prove Theorem \ref{th: converge of curves}.

\begin{proof}[Proof of Theorem \ref{th: converge of curves}]
By the Lemma \ref{lem: convergence curvature} we know that the Cheeger-Gromov convergence can be extended to the case of Riemann curvature tensors as well. The coefficients of the branching curve are given by polynomials of components of $\Rm$. By the elemantary fact that a polynomial is a continuous function the result follows.
\end{proof}

We demonstrate now the proof of Corollary \ref{cor: ilias}.

\begin{proof}[Proof of Corollary \ref{cor: ilias}]
By the Lemma \ref{lem: scale invariance} the branching curves are invariant under scalings of the metric by a constant factor. Thus $\Gamma_{x}^{g_i(t)}=\Gamma_{x}^{g(T+\lambda_i t)}$. By the Compactness Theorem of Cheeger-Gromov-Hamilton, there exists a subsequence $\{j_i\}$ such that $(M^4,g_{j_i}(t),x,F_{j_i}(t))$ converges to a complete, pointed ancient solution to the Ricci flow $(M^{4}_{\infty},g_{\infty}(t),x_{\infty},F_{\infty}(t))$ on $(-\infty,0)$. By the result of Enders-M\"uller-Topping (\cite{EMT}, Theorem 1.4) this singularity model is given by a nontrivial normalized gradient shrinking Ricci soliton in canonical form. The result follows immediately from Theorem \ref{th: converge of curves}.
\end{proof}

\begin{rem}
We strongly believe, that by choosing the $K3$ surface as an invariant instead of the branching curve, we can obtain even better results. The reason is, that  the $K3$ surfaces approach is more a sophisticated tool and their moduli space is well understood. Recall, that the interested reader can find more details on the coarse moduli space for lattice polarized $K3$ surfaces in the Appendix of \cite{Tergiakidis}. This will be part of our forthcoming work. The hope is, that these invariants will provide us with a better understanding of the generic singularity models for Type I singularities for the four dimensional Ricci flow.
\end{rem}

\end{document}